%% file: ABC.tex
\documentclass[11pt,reqno]{amsproc}
\input{header}

\begin{document}

\title[3$d$ exponential mixing and ideal dynamo with randomized ABC flows]{Three-dimensional exponential mixing and  ideal kinematic dynamo \\ with randomized ABC flows}

\author[M. Coti Zelati]{Michele Coti Zelati}
\address{(MCZ) Department of Mathematics, Imperial College London, London, SW7 2AZ, UK}
\email{m.coti-zelati@imperial.ac.uk}

\author[V. Navarro-Fernández]{Víctor Navarro-Fernández}
\address{(VNF) Department of Mathematics, Imperial College London, London, SW7 2AZ, UK}
\email{v.navarro-fernandez@imperial.ac.uk}

\subjclass[2020]{35Q49, 37A25, 76F25, 76W05}

\keywords{ABC flows, exponential mixing, ideal kinematic dynamo}

\date{\today}

\begin{abstract}
In this work we consider the Lagrangian properties of a  random version of the Arnold--Beltrami--Childress (ABC) in a three-dimensional periodic box.
We prove that the associated flow map possesses a positive top Lyapunov exponent and its associated one-point, two-point and projective Markov chains are geometrically ergodic. 
For a passive scalar, it follows that such a velocity is a space-time smooth exponentially mixing field, uniformly in the diffusivity coefficient. For a passive vector, it provides an example of a universal ideal (i.e. non-diffusive) kinematic dynamo.
\end{abstract}

\maketitle

\tableofcontents

\section{Introduction}

The study of turbulent and chaotic dynamical systems holds significant interest in both mathematics and physics due to its ability to uncover the inherent unpredictability and complexity found in various natural phenomena, such as weather patterns and fluid behavior. Serving as a quintessential example of a simple flow with intricate dynamics, this paper examines the \emph{Arnold--Beltrami--Childress} (ABC) flows. 
These flows are characterized by a set of parameters $\A,\B,\C\in\R$ and are defined on the three-dimensional torus $\T^3=\R^3/(2\pi\Z)^3$ by the vector field
\begin{equation}\label{eq:abc}
u(\xvec) =
\begin{pmatrix}
\A\sin z + \C\cos y \\
\B\sin x+ \A\cos z \\
\C\sin y + \B\cos x 
\end{pmatrix} \in\R^3,
\end{equation}
where we write $\xvec = (x,y,z) \equiv (x\mod 2\pi, y\mod 2\pi, z\mod 2\pi)\in\T^3$.

This class of vector fields possesses numerous intriguing properties that render them relevant for various applications in mathematical physics. Firstly, they are divergence-free, meaning $\nabla\cdot u = 0$. Secondly, they satisfy the so-called \emph{Beltrami} property, which means that there holds the identity $u=\nabla\times u$ for all $\A,\B,\C\in\R$. In particular, they define a stationary solution to the 3$d$ Euler equations for an inviscid fluid.
Since the work of Arnold \cite{Arnold65}, ABC flows have been considered as extraordinary examples of simple flows with a complicated topology and Lagrangian structure. If one of the parameters $\A$, $\B$ or $\C$ is zero, the flow is integrable. However, if all three parameters are nonzero, the behavior of the streamlines becomes complex and unpredictable, indicative of chaos \cites{DFGHMS86,Henon66}.

Another significant context in which the ABC flows arise is magnetohydrodynamics, particularly when studying how the movement of a magnetic fluid can alter the intensity of the magnetic field itself. This is related to the \emph{dynamo effect}, which is of particular interest in fields such as geophysics and astrophysics for explaining phenomena like stellar magnetism \cite{BrunBrowning17}. Childress \cite{Childress70} proposed the ABC flows as potential candidate vector fields that could cause the magnetic field's intensity to grow rapidly.

In this paper, we will study how the Lagrangian structure of a randomized version of \eqref{eq:abc} exhibits chaotic behavior. Specifically, this has significant consequences from the point of view of mixing of passive scalars and ideal dynamo in the context of ideal magnetohydrodynamics.

\subsection{Mixing of passive scalars}\label{section:intro-mixing}
The mixing properties of a divergence-free vector field $u:[0,\infty)\times\T^3\to\R^3$ are typically understood in terms of its ability 
to transfer the energy of a passive tracer from large to small spatial scales. Let $\kappa\in [0,1]$ be a diffusion parameter, and consider the Cauchy problem
\begin{equation}\label{eq:adv-diff}
\begin{cases}
\partial_t \rho + u\cdot \nabla \rho  =  \kappa\Delta\rho, \quad  &\text{in }(0,\infty)\times\T^3,\\
\rho(0,\cdot)  =  \rho_0, & \text{in }\T^3.
\end{cases}
\end{equation}
The scalar function  $\rho:[0,\infty)\times\T^3\to\R$ can represent, for instance, the concentration of some chemical in a solution, or the temperature of a gas in a room,
starting at a suitable, mean-free initial datum $\rho_0$.

The effect of a divergence free-vector field can create complicated and chaotic structures in the evolution of the passive scalar, see Figure \ref{fig:alternating-shear}, where the filamentation and the decreasing typical length-scale that can be appreciated is symptomatic of mixing. One way to quantify mixing for \eqref{eq:adv-diff}
is via the so-called \emph{functional mixing scale}, that involves the decay of negative Sobolev norms, defined through the Fourier transform $\widehat \rho$ by
\begin{equation}\label{eq:negSob}
\|\rho(t)\|_{\dot{H}^{-s}}^2 = \sum_{\kvec\in\Z^3_0} \frac{|\widehat{\rho}(t,\kvec)|^2}{|\kvec|^{2s}}, \qquad s>0, \qquad  \Z^3_0 = \Z^3\setminus\{(0,0,0)\}.
\end{equation}
The use of negative Sobolev norms is enlightening about the cascading mechanism to higher frequencies (i.e. small scales). The conservation of the $L^2$ norm implies that the sum of all modes $|\widehat{\rho}(t,\kvec)|^2$ must be constant in time, therefore if the $H^{-s}$ norm of $\rho(t)$ converges to zero as $t\to\infty$, it must occur that the mass concentrated in the lower modes ($|\kvec|$ small) is transported to the higher modes ($|\kvec|$ large). 

\begin{figure}[ht]
    \centering
    \includegraphics[width=\textwidth]{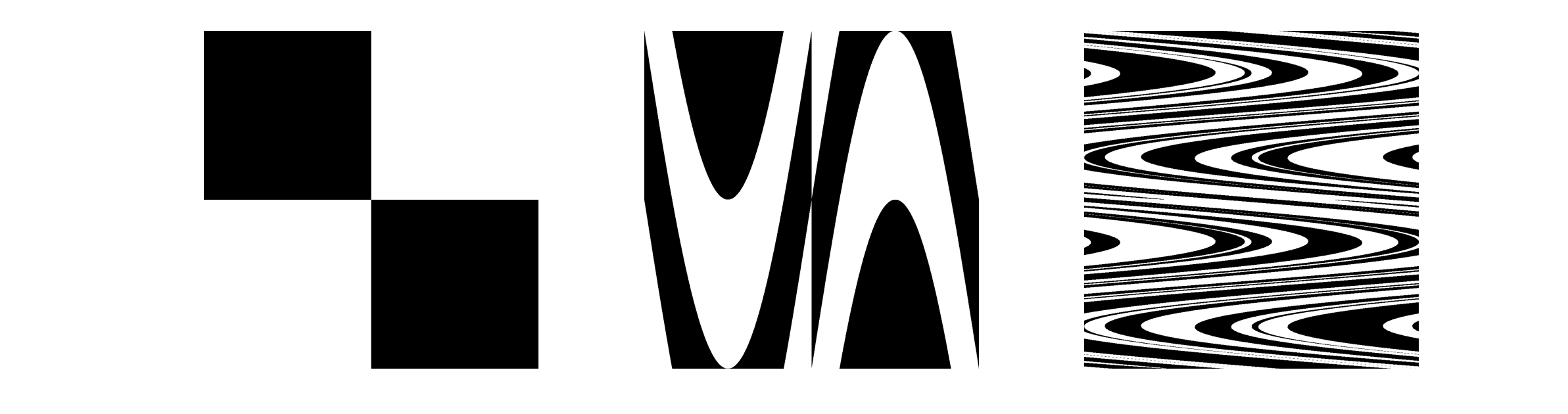}
    \caption{Mixing by alternating shear flows in $\T^2$}
    \label{fig:alternating-shear}
\end{figure}

This definition of mixing was first introduced in \cite{MathewMezicPetzold05}, and since then there have been a number of results and applications, some relevant examples of which are \cites{AlbertiCrippaMazzucato19a,ElgindiZlatos19,LunasinLinNovikovMazzucatoDoering12}. Other more geometrical definitions of mixing that under special circumstances might be equivalent to the one adopted here have been introduced to address these problems. For a review about the different ways of quantifying mixing, see \cite{Thiffeault12}.

It is well known that for sufficiently regular vector fields, mixing cannot be faster than exponential. This was proved for $u\in L^1_tW^{1,p}_{\xvec}$, $p\in (1,\infty]$, in \cite{CrippaDeLellis08} (see also \cites{IyerKiselevXu14,Seis13a}). Finding such lower bound for $p=1$ is related to Bressan's conjecture \cite{Bressan03} and remains open as of today. 

In this work we are interested in almost sure exponentially fast mixing in the context of randomized vector fields, where $u=u(t,\xvec,\omega)$ depends additionally on a random parameter $\omega$. Some examples of almost-sure exponential mixers have been recently provided via vector fields that are solutions to stochastically forced Navier--Stokes equations \cites{BBPS22a,BBPS22}, and via alternating shear flows in $\T^2$ with random phases \cite{BCZG23}, or with random switching times \cite{Cooperman23}. We provide an examples of a smooth, bounded exponentially mixing field in this article (see Theorem \ref{thm:mixing}).

The ideas presented in \cite{BCZG23} in the inviscid case $\kappa=0$, have been very recently extended in \cite{CIS24+} to the case of passive scalars that solve the advection-diffusion equation \eqref{eq:adv-diff} with $\kappa>0$. This means that apart from the transport by the vector field, also the molecular diffusion of the passive scalar is taken into account. In \cite{CIS24+}, the decay of the mixing norm is proved to be \emph{uniformly in the diffusion parameter}, so that the case without diffusion can be recovered as well. We will present our results using this strictly more general framework.

\subsection{Dynamos in passive vectors}\label{sub:passivevectors}
The evolution of a three-dimensional magnetic field $B=B(t,\xvec)$ in a homogeneous moving conductor is described by a combination of Maxwell's equations and Ohm's law, yielding the set of equations
\begin{equation}\label{eq:dynamo}
\begin{cases}
\partial_t B + (u\cdot\nabla) B - (B\cdot \nabla) u  =  \kappa \Delta B ,\quad  &\text{in }(0,\infty)\times\T^3, \\
\nabla\cdot B  =  0, & \text{in }(0,\infty)\times\T^3, \\
B(0,\cdot)  =  B_0, & \text{in }\T^3.
\end{cases}
\end{equation}
Here, $u$ is the velocity field of the moving conductor, $\kappa\geq 0$ is the magnetic resistivity of the medium, and $B_0$ an assigned, mean-free initial condition.

Solutions of \eqref{eq:dynamo} can behave quite differently compared to those of \eqref{eq:adv-diff}. If $u$ is Lipschitz, a standard stability estimate and the Poincaré inequality give that the $L^2$ norm of $B$ cannot grow faster than exponential, due to the energy estimate
\[
    \|B(t)\|_{L^2} \leq \|B_0\|_{L^2} \e^{(\|\nabla u\|_{L_{t,x}^\infty}-\kappa)t}.
\]
The potential growth of the $L^2$ norm is due to the  stretching term $(B\cdot\nabla)u$. 
Using the terminology of \cite{ChildressGilbert}, we define different dynamo behaviors in terms of the asymptotic growth rate
\begin{equation}
\chi(\kappa):= \sup_{B_0\in L^2} \limsup_{t\to\infty} \frac1t \log\|B(t)\|_{L^2},
\end{equation}
uniformly with respect of initial data in $L^2$:
\begin{itemize}
    \item \emph{Ideal dynamo}: $\chi(0)>0$. This involves only the non-resistive case, when we set $\kappa=0$ in \eqref{eq:dynamo}, and postulate the existence of a velocity field that exhibits exponential growth of the magnetic field. We will provide an example below (see Theorem \ref{thm:fast-dynamo});
    \item \emph{Fast dynamo}: $\chi_0:=\displaystyle \liminf_{\kappa\to0}\chi(\kappa)>0$. Informally, the exponential growth rate of $\|B(t)\|_{L^2}$ becomes independent of $\kappa$ in the perfectly conducting limit, and whether this happens is an outstanding open problem, formulated by Y.B. Zeldovich in 1957 in \cite{zil1957magnetic}, and mentioned in the Arnold's problem monograph  \cite{Arnold04}*{1994-28} (see also \cite{ArnoldKhesin}*{Chapter V} or \cite{ChildressGilbert}).
    \item \emph{Slow dynamo}: $\chi_0:=\displaystyle \liminf_{\kappa\to0}\chi(\kappa)\leq 0$. In this case, the exponential growth is suppressed as $\kappa\to 0$. Evidence of this behavior (with $\chi(\kappa)\sim \kappa^\frac13$) is provided by the so-called Ponomarenko dynamo \cites{Ponomarenko73,gilbert1988fast}.
\end{itemize}
Besides being interesting on its own from a mathematical perspective, this problem has relevant applications to astrophysics and geophysics, in particular to stellar magnetism \cites{Childress70,KumarRoberts75}.
The intuitive idea behind a fast dynamo is the \emph{Stretch--Twist--Fold} mechanism \cite{ChildressGilbert}, which displays a discrete in time flow that \emph{doubles} the density of magnetic lines on each iteration. ABC flows were proposed by Childress \cite{Childress70} as suitable candidates to be a fast dynamo in 3$d$, as their streamlines appear to be chaotic \cites{Arnold65,DFGHMS86,ZKH93}. This claim is also supported by a strong numerical evidence \cites{Alexakis11,IsmaelEmmanuel13,JonesGilbert14}.

\subsection{The Lagrangian picture and main results}\label{subsection:definition-flow}

The two problems described in Sections \ref{section:intro-mixing} and \ref{sub:passivevectors} share the same Lagrangian picture, in terms of the associated flow map 
\begin{equation}\label{eq:sde-diffusion}
    \dd\phi_t(\xvec) = u(t,\phi_t(\xvec))\dd t + \sqrt{2\kappa}\dd W_t, \quad \phi_0(\xvec)=\xvec.
\end{equation}
where $W_t$ represents a standard Brownian motion with periodic boundary conditions in $\T^3$. 
Using the Feynman-Kac formula, solutions to \eqref{eq:adv-diff} are given by 
$$
\rho(t,\xvec) = \EX_W\left[(\phi_t)_\#\rho_0(\xvec)\right],
$$
while for \eqref{eq:dynamo} we have
\begin{equation}\label{eq:FKpassiveV}
B(t,\xvec) = \EX_W\left[(\phi_t)_\#D_{\xvec} \phi_t B_0(\xvec)\right].
\end{equation}
In both cases above, the expectation is taken with respect to the noise associated to the Brownian motion, and it is therefore denoted by $\EX_W$. Nonetheless, this may not be the only source of randomness, as $u$ may contain some stochasticity itself. This is indeed the point of view we take here.

Let $U>0$ be a fixed real number, and consider two collections of random variables $(\A,\B,\C)\in [-U,U]^3\subset\R^3$, $(\alpha,\beta,\gamma)\in[0,2\pi)^3\subset\R^3$ which are independent and identically distributed (IID). These variables represent random amplitudes and random phases respectively that will be inserted in the ABC vector field \eqref{eq:abc}. We define the probability space where the random variables take values as $(\Omega_0,\F_0,\Prob_0)$, with 
\[
\Omega_0 = [-U,U]^3\times[0,2\pi)^3, \quad \F_0=\Bo([-U,U]^3\times[0,2\pi)^3),
\]
and $\Prob_0$ is the uniform probability measure in $\Omega_0$.

The $(\A,\alpha)$, $(\B,\beta)$ and $(\C,\gamma)-$flows are defined respectively by
\begin{align*}
f_{(\A,\alpha)}(x,y,z) = 
\begin{pmatrix} 
x + \A\sin (z+\alpha)  \\
y + \A\cos (z+\alpha)  \\
z  
\end{pmatrix} ,\\
f_{(\B,\beta)}(x,y,z) =\begin{pmatrix} 
x  \\
y + \B\sin (x+\beta)  \\
z + \B\cos (x+\beta)
\end{pmatrix},\\
f_{(\C,\gamma)}(x,y,z) = \begin{pmatrix} 
x + \C\cos (y+\gamma)  \\
y   \\
z + \C\sin (y+\gamma) 
\end{pmatrix}.
\end{align*}
At the $i$-th iteration (without loss of generality we take it to be of duration 1), we make a random choice of parameters $\omega_i=(\A_i,\B_i,\C_i,\alpha_i,\beta_i,\gamma_i)\in\Omega_0$ and apply the maps $f_{(\cdot,\cdot)}$ to obtain the composition
\[
f_{\omega_i}(x,y,z) = \left(f_{(\C_i,\gamma_i)}\circ f_{(\B_i,\beta_i)}\circ f_{(\A_i,\alpha_i)}\right)(x,y,z).
\]
The corresponding velocity vector $u$, which depends on the noise path $\underline{\omega}=(\omega_1,\omega_2,\hdots)\in\Omega = \Omega_0^\N$, is given by 
\begin{equation}\label{eq:random-abc-vectorfield}
\begin{aligned}
u (t,\xvec,\underline{\omega}) & = \displaystyle \sum_{i=0}^\infty  
\begin{pmatrix}
    \A_i \sin (z+\alpha_i)  \\
    \A_i \cos(z+\alpha_i) \\
    0
\end{pmatrix}
  \chi_{\left[3i,3i+1\right)}(t) + \sum_{i=0}^\infty 
\begin{pmatrix}
    0  \\
    \B_i \sin(x+\beta_i) \\
    \B_i \cos(x+\beta_i)
\end{pmatrix}
\chi_{\left[3i+1,3i+2\right)}(t) \\ 
& \displaystyle \quad + \sum_{i=0}^\infty 
\begin{pmatrix}
    \C_i \cos(y+\gamma_i)  \\
    0 \\
    \C_i \sin(y+\gamma_i)
\end{pmatrix}
 \chi_{\left[3i+2,3i+3\right)}(t),
\end{aligned}
\end{equation}
with $\chi_{[a,b)}(t)$ denoting the characteristic function of an interval $[a,b)$ in the real line. Observe that even though this vector field is not smooth in $t$, it can be turned into a smooth vector field with the addition of a suitable mollifying ``switching'' function.
 
In the time discrete setting presented here, the flow $\phi_t$ is recovered for $t=3n\in\N$ by
\begin{equation}\label{eq:ABCflow-solution}
\phi_{3n}(\xvec) = f_{\underline{\omega}}^n(\xvec) = f_{\omega_n}\circ f_{\omega_{n-1}}\circ \hdots \circ f_{\omega_1}(\xvec).
\end{equation}
The main results of this article are that $u$ is an exponentially mixing vector field on $\T^3$, uniformly in the diffusivity parameter $\kappa\geq 0$, and is also an example of an ideal dynamo. We begin by stating the result on passive scalars.

\begin{theorem}\label{thm:mixing}
Let $(\phi_t)_{t\geq 0}$ be the flow defined by the SDE \eqref{eq:sde-diffusion} with $\kappa\geq 0$ and $u$ the random ABC vector field \eqref{eq:random-abc-vectorfield}. For any $q,s>0$ there exist a random constant $\widehat{D}_{\underline{\omega},\kappa}\geq 0$ and a deterministic, $\kappa-$independent constant $\lambda_s>0$ such that for all mean free $g,h\in H^s(\T^3)$, there holds
\[
\left| \int_{\T^3} g(\xvec)h(\phi_t(\xvec))\dd \xvec \right| \leq \widehat{D}_{\underline{\omega},\kappa}\|g\|_{\dot{H}^s} \|h\|_{\dot{H}^s}\e^{-\lambda_s t},
\]
almost surely for all $t>0$. Moreover, there exists a $\kappa-$independent constant $\overline{D}_q>0$ such that
\[
\EX_\Prob |\widehat{D}_{\cdot,\kappa}|^q\leq \overline{D}_q.
\]
\end{theorem}

The proof of this theorem is presented in Section \ref{section:proof-mixing}, and it is based on ideas first introduced in \cite{BCZG23} for the purely transport case $\kappa=0$, and recently extended to the advection-diffusion equation $\kappa>0$ in \cite{CIS24+}.

The result in Theorem \ref{thm:mixing} gives a decay of the so-called \emph{correlations} between any $H^s$ and mean free functions $g$ and $h$. This is a standard definition of mixing in ergodic theory, and it yields the following consequences for the passive scalar $\rho$ that solves \eqref{eq:adv-diff}.

\begin{corollary}\label{corollary:mixing-dissipation}
Let $\rho$ be a solution to the advection-diffusion equation \eqref{eq:adv-diff} transported by the random ABC vector field \eqref{eq:random-abc-vectorfield} with initial configuration $\rho_0\in H^s(\T^3)$ and mean free. Using the notation from Theorem \ref{thm:mixing}, we obtain the following estimates.
\begin{itemize}
    \item Exponential mixing $(\kappa\geq 0)$: For any $s>0$, 
    \begin{equation}\label{eq:scalmix}
      \|\rho(t)\|_{\dot{H}^{-s}} \leq \widehat{D}_{\underline{\omega},\kappa}\|\rho_0\|_{\dot{H}^s}\e^{-\lambda_s t},      
    \end{equation}
    for every $t\geq0$.
    \item Enhanced dissipation $(\kappa>0)$: For any $s>0$,  
    \begin{equation}\label{eq:scalenhanced}
    \|\rho(t)\|_{L^2} \leq \frac{\widehat{D}_{\underline{\omega},\kappa}}{\kappa^{\frac{3}{2}+s}}\|\rho_0\|_{L^2}\e^{-\lambda_s t},      
    \end{equation}
    for every $t\geq0$.
\end{itemize}
\end{corollary}

The decay of correlations from Theorem \ref{thm:mixing} implies the decay of all $\dot{H}^{-s}$ norms (mixing)in \eqref{eq:scalmix} via duality. In particular, Corollary \ref{corollary:mixing-dissipation} shows that \eqref{eq:random-abc-vectorfield} is an example of a smooth almost-sure exponential mixer in $\T^3$, which is the fastest possible mixing rate for Lipschitz vector fields.

Corollary \ref{corollary:mixing-dissipation} also showcases that the vector field \eqref{eq:random-abc-vectorfield} enhances dissipation, see \eqref{eq:scalenhanced}. If $\kappa>0$, mean free solutions to advection-diffusion equations with divergence free vector fields display an exponential decay of their $L^2$ norm. In general, the effect of diffusion makes the $L^2$ norm of the passive scalar to be halved in a timescale $O(\kappa^{-1})$. On the other hand, the presence of a mixing velocity $u$ enhances dissipation \cites{ConstantinKiselevRyzhikZlatos08,CotiZelatiDelgandioElgindi20}, and the example of Corollary \ref{corollary:mixing-dissipation} shows that the halving time of $\|\rho(t)\|_{L^2}$ is $O(|\log\kappa|)$ -- which is \emph{optimal} at this regularity level  \cite{Seis23}. For a proof of the enhanced dissipation estimate given uniform--in--$\kappa$ decay of correlations, and other examples see \cites{BBPS21,CIS24+,TZ23}.

For passive vectors, we dwell with the non-resistive setting $\kappa=0$. 

\begin{theorem}\label{thm:fast-dynamo}
Let $\kappa=0$, $p\in [1,\infty]$, and consider an initial datum $0\neq B_0\in L^p(\T^3)$. Then there exist a deterministic constant $\eta>0$, and a $\Prob-$almost surely positive random constant $\hat{k}_{\underline{\omega}}>0$ such that the solution $B(t,\xvec)$ to \eqref{eq:dynamo} satisfies
\[
\|B(t)\|_{L^p} \geq \hat{k}_{\underline{\omega}}\e^{\eta t},
\]
for all $t\geq 0$. In particular, the random ABC vector field $u$ \eqref{eq:random-abc-vectorfield} is an ideal dynamo with probability 1.
\end{theorem}

The idea of considering stochastic vector fields appears naturally from the necessity of studying objects like the Lyapunov exponents of the flow, and it has been proposed before with e.g.\ a Gaussian random field \cite{BaxendaleRozovskii93}, or the ABC flow itself with fluctuating coefficients \cite{TominSokoloff10}. The result in Theorem \ref{thm:fast-dynamo}, obtained from the techniques applied for Theorem \ref{thm:mixing}, is to the best of our knowledge the first rigorous proof of ABC flows being an ideal kinematic dynamo. The case with magnetic diffusivity remains an open problem.

\subsection*{Organisation of the paper} 

In Section \ref{section:harris} we introduce some concepts and preliminary results needed to prove Theorems \ref{thm:mixing} and \ref{thm:fast-dynamo}. In particular we define three Markov chains in terms of the flow generated by the vector field \eqref{eq:random-abc-vectorfield}, we present a suitable version of Harris Theorem about the ergodicity of the Markov chains, and we introduce some notions from the theory of random dynamics that will be key to prove our results. In Section \ref{section:proof-mixing} we prove the dynamical properties required for the random ABC vector field to be an exponential mixer. This includes showing topological irreducibility, aperiodicty and small sets of the different Markov chains, as well as the positivity of the Lyapunov exponent of the flow defined by $u$. Finally, in Section \ref{section:dynamo}, as a byproduct of the ergodic properties of the random flow, we prove Theorem \ref{thm:fast-dynamo}.

\section{Harris theorem and preliminary results}\label{section:harris}

In this section we introduce preliminary notions and results that will be important to prove Theorems \ref{thm:mixing} and \ref{thm:fast-dynamo}. First of all we introduce three different Markov transition kernels that describe the dynamics of the inviscid flow $f_\omega$ generated by the stochastic ABC vector field \eqref{eq:random-abc-vectorfield}.
\begin{enumerate}
    \item \emph{The one-point chain} in $\T^3$. 
    It describes the probability of a particle that started at position $\xvec\in \T^3$, to be in a Borel set $A\in\cB(\T^3)$ after one full iteration of the flow map,
    \[
    P(\xvec,A) = \Prob_0[f_\omega(\xvec)\in A].
    \]
    \item \emph{The projective chain} in $\T^3\times\S^2$. It describes the dynamical system defined the position $\xvec\in\T^3$ and direction of movement $\vvec\in\S^2$ of a single particle. For this consider a Borel set $\hat{A}\in\Bo(\T^3\times \S^2)$, then we define the projective Markov chain as
    \[
    \hat{P}((\xvec,\vvec),\hat{A})= \Prob_0\left[ \left(f_\omega(\xvec),\frac{D_{\xvec} f_\omega\vvec}{|D_{\xvec} f_\omega\vvec|}\right)\in \hat{A} \right].
    \]
    \item \emph{The two-point chain} in $(\T^3\times\T^3)\setminus\Delta$. Here $\Delta = \{(\xvec^1,\xvec^2)\in\T^3\times\T^3\mid \xvec^1=\xvec^2\}$ represents the diagonal in the product space $\T^3\times\T^3$. The two-point chain gives information about the position of two particles starting at $\xvec^1\in\T^3$ and $\xvec^2\in\T^3$ respectively, with $\xvec^1\neq\xvec^2$. With this notation in hand, given any Borel set $A^{(2)}\in\cB((\T^3\times\T^3)\setminus\Delta)$, and any $(\xvec^1,\xvec^2)\in(\T^3\times\T^3)\setminus\Delta$, we define the two-point Markov chain via
    \[
    P^{(2)}((\xvec^1,\xvec^2),A^{(2)})= \Prob_0\left[ \left(f_\omega(\xvec^1),f_\omega(\xvec^2)\right)\in A^{(2)} \right].
    \]
\end{enumerate}

In the following section we will introduce some basic theory about the Markov trasition kernel that will be applicable for the three different chains: one-point, projective and two-point. We present the results in a more abstract setting so that they are applicable in for the three of them. To do so, we denote the ambient space by $\X$, and we assume it to be a complete metric space.

\subsection{Markov chains and Harris theorem}

Let $P$ be a Markov transition kernel on $\X$. In what comes next we will denote by $\cB(\X)$ the set of Borel measurable sets on $\X$, and $\cP(\X)$ the set of probability measures on $\X$. In addition $\chi_A$ denotes the characteristic function of $A\in\cB(\X)$. With this notation in mind, we introduce the following features of Markov transition kernels.
\begin{itemize}
\item $P(\xvec,\cdot)$ is a probability measure on $\X$ for each $\xvec\in\X$.
\item Iterations of the Markov transition kernel are defined via Chapman-Kolmogorov,
\[
P^{n+1}(\xvec,A) = \int_{\X} P^n(\zvec,A)P(\xvec,\dd \zvec),
\]
for all $n\in\N$ and Borel $A\in\cB(\X)$.
\item $P$ acts on bounded functions $h:\X\to\R$ via
\[
Ph(\xvec) = \int_{\X}h(\zvec)P(\xvec,\dd \zvec).
\]
If $Ph(\xvec)$ is a continuous and bounded function on $\X$ for any $h(\xvec)$ continuous and bounded, we say that $P$ has the \emph{Feller property}.
\item $P$ acts on probability measures $\mu$ on $\X$ via
\[
P\mu(A) = \int_{\X} P(\xvec,A)\dd \mu(\xvec),
\]
for any Borel $A\in\cB(\X)$.
\end{itemize}

\begin{definition}
We say that a probability measure $\mu\in\cP(\X)$ is \emph{stationary} for the Markov transition kernel $P$ if it is a fixed point for its action on the space of probability measures, namely $\mu(A)=P\mu(A)$ for any $A\in\cB(\X)$.
\end{definition}

\begin{definition}
We say that a probability measure $\mu\in\cP(\X)$ is \emph{ergodic stationary} for the Markov transition kernel $P$ if all Borel sets $A\in\cB(\X)$ for which $P\chi_A(\xvec)=\chi_A(\xvec)$ $\mu-$a.e.\ satisfy $\mu(A)=0$ or $\mu(A)=1$.
\end{definition}

It is a classical result that if a Markov transition kernel $P$ has a unique stationary measure $\mu$, then $\mu$ is ergodic stationary, see for instance \cite{DaPratoZabczy96}*{Theorem 3.2.6}.

Next, we introduce a version of Harris theorem that is convenient of our purposes in this paper. In order to do so, let us define the following space of weighted bounded functions. Given any function $V:\X\to [1,\infty)$, we say that $\phi\in L^\infty_V(\X)$ if $\phi$ is measurable and
\[
\|\phi\|_{L^\infty_V} = \sup_{\xvec\in\X}\frac{|\phi(\xvec)|}{V(\xvec)}<\infty.
\]
For a proof of this particular version of this theorem we refer to \cite{BCZG23}*{Appendix A}. For a proof of more general and other versions we refer to \cites{HairerMattingly11,MeynTweedieGlynn09}.
\begin{theorem}[Harris theorem]\label{thm:harris}
Let $P$ be a Markov transition kernel on $\X$ with the Feller property that possesses the following attributes.
\begin{enumerate}[label=(\arabic*), ref=(\arabic*)]
\item (Topological irreducibility) For all $\xvec\in\X$ and all $A\subset \X$ open and nonempty, there exists $n\in\N$ such that $P^n(\xvec,A)>0$.
\item (Small set) There exists $n\in\N$, an open set $A\subset\X$, and a positive measure $\nu_n$ in $\X$ such that for all $\xvec\in A$ there holds
\[
P^n(\xvec,B) \geq \nu_n(B)
\]
for all $B\in\cB(\X)$.
\item (Aperiodicity) There exists $\xvec^\star\in\X$ such that for all $A\in\X$ open and containing $\xvec^\star$ there holds that $P(\xvec^\star,A)>0$.
\item\label{item:drift} (Lyapunov drift condition) There exists a function $V:\X\to[1,\infty)$ with the following property: there exist $\alpha\in (0,1)$, $\beta>0$, and a compact set $K\subset\X$ such that
\[
PV(\xvec) \leq \alpha V(\xvec)+\beta\chi_K(\xvec)
\]
for all $\xvec\in\X$.
\end{enumerate}
Then $P$ is $V-$uniformly geometrically ergodic, namely there exists a unique stationary measure $\mu$ for $P$, and there exist $C>0$ and $\gamma\in(0,1)$ such that
\[
\left| P^n\phi(\xvec)-\int_{\X}\phi(\xvec)\dd\mu(\xvec) \right| \leq CV(\xvec)\|\phi\|_{L^\infty_V}\gamma^n
\]
for all $\xvec\in\X$, all $\phi\in L^\infty_V(\X)$ and all $n\in\N$.
\end{theorem}

If $\X$ is compact, then the Lyapunov drift condition \ref{item:drift} is redundant, since it will always be satisfied for the trivial function $V(\xvec)=1$. This is due to the fact that $\|PV\|_{L^\infty}<\infty$ for any $V\in L^\infty$. Choosing $K=\X$ (compact), $\beta = \|PV\|_{L^\infty}=1$ and any $\alpha\in(0,1)$ suffices to obtain the result. In the case of $\X$ compact, we will say that a Markov--Feller transition kernel $P$ is \emph{uniformly geometrically ergodic} provided that it satisfies topological irreducibility, small sets and aperiodicity conditions from Harris theorem.

\subsection{Random dynamics and Lyapunov exponents}\label{subsection:random-dynamics}

The Markov processes introduced in the previous section are derived from continuous \emph{random dynamical systems} (RDS for short) with independent increments. For a reference about the theory of random dynamical systems see the monograph by Arnold \cite{Arnold98}. 

Consider $(\Omega_0,\cF_0,\Prob_0)$ a fixed probability space, and denote its elements by $\omega\in\Omega_0$. Following the notation of Section \ref{subsection:definition-flow} we introduce the composition 
\[
f_{\underline{\omega}}^n(\xvec) = f_{\omega_n}\circ \hdots \circ f_{\omega_1}(\xvec),
\]
where $\underline{\omega} = (\omega_1,\omega_2,\hdots)$ are elements of the product space $\Omega = \Omega_0^\N$. In a similar fashion, we define the product probability space $(\Omega,\cF,\Prob) = (\Omega_0,\cF_0,\Prob_0)^\N$. In addition, let $\theta:\Omega\to\Omega$ be the leftwards shift, defined by $\theta\underline{\omega}=(\omega_2,\omega_3,\hdots)$. Recall that the leftward shift is a measure preserving transformation on $(\Omega,\cF,\Prob)$, i.e.\ $\theta_\#\Prob = \Prob$, and therefore $(\Omega,\cF,\Prob,\theta)$ defines a measure-preserving dynamical system. We introduce the following definition.
\begin{definition}
    We say that $f_{\underline{\omega}}^n$ is a \emph{continuous RDS} over $(\Omega,\cF,\Prob,\theta)$ if there holds:
    \begin{enumerate}
        \item the mapping $f_\omega:\X\to\X$ is continuous for all $\omega\in\Omega_0$;
        \item the set $\{\omega\in\Omega_0 \mid f_\omega(\xvec)\in A\}$ is $\cF_0-$measurable for all $\xvec\in\X$ and all $A\in\cB(\X)$;
        \item $f^0_{\underline{\omega}}$ is the identity map in $\X$ for all $\underline{\omega}\in\Omega$;
        \item $f_{\underline{\omega}}^n$ satisfies the \emph{cocycle property}, namely
        \[
        f_{\underline{\omega}}^{n+m}(\xvec) = f_{\theta^m\underline{\omega}}^n\circ f_{\underline{\omega}}^m(\xvec)
        \]
        for all $n,m\in\N$ and all $\underline{\omega}\in\Omega$.
    \end{enumerate}
\end{definition}

Markov processes defined via continuous RDS $P(\xvec,A) = \Prob_0[f_\omega(\xvec)\in A]$ automatically have the Feller property due to the continuity of $f_\omega:\X\to\X$.

One key ingredient to study the dynamics and long-time behaviour of a system defined by a RDS $f_\omega$ is the matrix $D_{\xvec}f_{\underline{\omega}}^n$, that gives information about the change of position of a particle at each time step. This object can be recursively defined from the measurable mapping $D_{\xvec}f_\omega:\Omega_0\times\X\to GL_d(\R)$, where $d=\dim\X$, via the composition 
\[
D_{\xvec}f_{\underline{\omega}}^n = D_{f_{\underline{\omega}}^{n-1}(\xvec)}f_{\omega_n} \circ\hdots\circ D_{\xvec}f_{\omega_1}.
\]
Elements with this property are called \emph{linear cocycles}. The long-time behaviour of linear cocycles is described by the \emph{Multiplicative Ergodic Theorem} (MET), see \cite{Arnold98}*{Part II}, for linear cocycles that satisfy the integrability condition
\begin{equation}\label{eq:MET_integrability}
    \int_{\Omega_0\times\X} (\log^+|D_{\xvec}f_\omega| + \log^+|(D_{\xvec}f_\omega)^{-1}|) \dd\Prob_0(\omega)\dd\mu(\xvec) < \infty,
\end{equation}
where $\log^+ a = \max\{\log a,0\}$, and $\mu$ is a stationary measure of the Markov chain generated by the continuous RDS $f_\omega$. The MET states that there exists $r\in\{1,\hdots,d\}$, a sequence $\lambda_1(\omega,\xvec)>\hdots>\lambda_r(\omega,\xvec)$, and a filtration of subspaces
\[
\R^d = F_1(\omega,\xvec) \supsetneq \hdots \supsetneq F_r(\omega,\xvec) \supsetneq F_{r+1}(\omega,\xvec)=\{0\},
\]
with the property
\[
\lambda_i(\omega,\xvec) = \lim_{n\to\infty} \frac{1}{n}\log|D_{\xvec}f_{\underline{\omega}}^n \vvec|,
\]
for all $\vvec\in F_i(\omega,\xvec)\setminus F_{i+1}(\omega,\xvec)$. The subspaces $F_i(\omega,\xvec)$ are usualle referred to as \emph{Oseledets spaces}, the numbers $\lambda_i(\omega,\xvec)$ are called \emph{Lyapunov exponents}, and their multiplicities are defined by
\[
m_i = \dim F_i(\omega,\xvec)-\dim F_{i+1}(\omega,\xvec).
\]
In addition, the MET ensures that if the stationary measure $\mu$ of $f_\omega$ is ergodic, then the Lyapunov exponents $\lambda_i(\omega,\xvec) \equiv \lambda_i$ are constant for $\Prob_0-$a.a.\ $\omega\in\Omega_0$, and $\mu-$a.e.\ $\xvec\in\X$.

The Lyapunov exponents are key objetcs that play a crucial role in the proofs of Theorems \ref{thm:mixing} and \ref{thm:fast-dynamo}. An important property of these exponents is that its total sum is given by
\[
\lambda_\Sigma = m_1\lambda_1 + \hdots +m_r\lambda_r = \lim_{n\to\infty} \frac{1}{n}\log|\det (D_{\xvec}f_{\underline{\omega}}^n) |.
\]
This is of particular interest for RDS defined by incompressible flows because they are volume preserving and thus $|\det (D_{\xvec}f_{\underline{\omega}}^n)|=1$ for all $n\in\N$. With this in mind we see that the top Lyapunov exponent satisfies $\lambda_1\geq 0$, and $\lambda_1>0$ if $r\geq 2$.

A strictly positive Lyapunov exponent is sometimes defined as \emph{Lagrangian chaos}, see \cite{BBPS22a}. It is a feature that goes in favour of chaos and mixing for the passive scalar, and growth of the $L^2$ norm for the passive vector. To rule out the case in which the top Lyapunov exponent is zero we use the following criterion originally introduced by Furstenberg \cite{Furstenberg63}.

\begin{proposition}[Furstenberg's criterion]\label{prop:furstenberg}
    Assume that $f_\omega$ is a volume-preserving continuous RDS with ergodic stationary measure $\mu$ and that it satisfies \eqref{eq:MET_integrability}. If $\lambda_1=0$, then there exists a (measurable) family of measures $\{\nu_{\xvec}\mid \xvec\in\supp(\mu)\}$ on $\S^{d-1}$ such that
    \[
    (D_{\xvec}f_{\underline{\omega}}^n)_\# \nu_{\xvec} = \nu_{f_{\underline{\omega}}^n(\xvec)}
    \]
    for $\Prob-$a.a.\ $\underline{\omega}\in\Omega$, $\mu-$a.e.\ $\xvec\in\X$ and all $n\in\N$.
\end{proposition}

The version of Furstenberg's criterion here presented can be found in \cites{BBPS22a,BCZG23}. In the following section we will introduce sufficient conditions to rule out the case $\lambda_1=0$, which together with a suitable application of Harris theorem will yield the results in Theorems \ref{thm:mixing} and \ref{thm:fast-dynamo}.

\subsection{Preliminary results}

In \cites{BCZG23,CIS24+} the authors present a collection of sufficient conditions to check the assumptions in Harris Theorem \ref{thm:harris}. For the next results we will look at $f_\omega$ as a continuous random dynamical system on $\Omega_0\times\X\times\N$, and we will assume that it satisfies the following basic hypotheses.
\begin{enumerate} [label=(H\arabic*), ref=(H\arabic*)]
\item\label{item:H1} The IID noise parameters $\omega$ live in a probability space $\Omega_0=\R^6$ with probability measure $\dd\Prob_0(\omega) = \rho_0(\omega)\dd \omega$, where $\dd \omega$ represents the Lebesgue measure on $\R^6$, and $\rho_0(\omega)$ is a density. Additionally, the mapping $(\omega,\xvec)\mapsto f_\omega(\xvec)$ is $C^2(\Omega_0\times\X)$. 
\item\label{item:H2} There exist a constant $C_0>0$ such that  $\Prob_0-$a.s.
\[
|(D_{\xvec} f_\omega)^{-1}|\leq C_0, \quad |D_{\xvec} f_\omega|\leq C_0, \quad \text{and}\quad \|f_\omega\|_{C^2}\leq C_0.
\]
\item\label{item:H3} $f_\omega$ preserves the Lebesgue measure on $\X$ $\Prob_0-$a.s.
\end{enumerate}

The next result asserts that if the two-point chain generated by the inviscid flow is geometrically ergodic, then $u$ mixes almost surely any $H^s$ initial datum of the advection-diffusion equation \eqref{eq:adv-diff} exponentially fast, provided that the diffusivity $\kappa\geq 0$ is sufficiently small. This result for the case $\kappa=0$ was originally proved in \cite{BCZG23}*{Proposition 4.6} for a time continuous flow, and in \cite{BBPS22}*{Section 7} for a discrete framework more similar to the one presented here. Recently this mixing result has been extended to the more general case including diffusion in \cite{CIS24+}*{Lemma 3.3}.

\begin{proposition}\label{prop:discrete-mixing}
Let $f_\omega^\kappa$ be a continuous RDS on $\T^3$ defined as a solution to the SDE \eqref{eq:sde-diffusion} with $\kappa\geq 0$. Let the two-point process of $f_\omega^0$ be $V-$geometrically ergodic with $V\in L^1(\T^3\times\T^3)$. Then for any $q,s>0$ and any $\kappa\geq 0$ sufficiently small, there exist a random constant $\widehat{C}_{\underline{\omega},\kappa}\geq 1$, and a deterministic, $\kappa-$independent constant $\lambda_s>0$ such that for all mean free $g,h\in H^s(\T^3)$, there holds
\[
\left| \int_{\T^3} g(\xvec)h(f_{\underline{\omega}^n}^\kappa(\xvec))\dd \xvec \right| \leq \widehat{C}_{\underline{\omega},\kappa} \|g\|_{H^s} \|h\|_{H^s} \e^{-\lambda_s n} \quad \text{almost surely for all } n\in\N.
\]
Moreover, there exists a $\kappa-$independent constant $\overline{C}_q>0$ such that $\EX| \widehat{C}_{\cdot,\kappa}|^q\leq \overline{C}_q$.
\end{proposition}

Proposition \ref{prop:discrete-mixing} is the key piece to obtain Theorem \ref{thm:mixing}. It states that we only need to prove geometric ergodicity for the two-point process of the inviscid problem $\kappa=0$ to obtain mixing uniformly in $\kappa$. Therefore from now on we will drop the superscript $\kappa$ in $f^\kappa_\omega$ and we will address exclusively the inviscid case: $f_\omega$ denotes the flow given by the vector field \eqref{eq:random-abc-vectorfield} evaluated at discrete times.

Since the goal is to prove ergodicity of the two-point chain $P^{(2)}$, we need to make sure that assumptions in Harris theorem are satisfied. The main obstacle will be to show the Lyapunov-drift condition, which is required because $P^{(2)}$ is defined in a non-compact space $\X=(\T^3\times\T^3)\setminus\Delta$, however irreducibility, small sets and aperiodicity must also be checked. Topological irreducibility for the case of the random ABC vector field will be a byproduct of proving controllability. Here we present sufficient conditions for small sets and aperiodicity. The first result, about the existence of a small set, is proved in \cite{BCZG23}*{Proposition 3.1}.

\begin{lemma}[Sufficient condition for small sets]\label{prop:small-set}
Assume that $f_\omega$ satisfies \ref{item:H1}. Assume that there exist $n\in\N$, $\underline{\omega}^n_\star\in\Omega_0^n$, and $\xvec^\star\in\X$ that satisfy:
\begin{enumerate}
\item $\exists \varepsilon,\delta>0$ so that $\rho_0(\underline{\omega}^n)\geq \delta$ for all $\underline{\omega}^n\in\Omega_0^n$ with $|\underline{\omega}^n-\underline{\omega}_\star^n|\leq \varepsilon$;
\item $\Phi_{\xvec^\star}(\underline{\omega}^n) = f_{\underline{\omega}^n}(\xvec)$ is a submersion at $\underline{\omega}^n=\underline{\omega}^n_\star$.
\end{enumerate}
Then there exist $A\subset\X$ open, and $\nu_n$ positive and absolutely continuous with respect to Lebesgue such that for all $\xvec\in A$ there holds
\[
P^n(\xvec,B) \geq \nu_n(B)
\]
for all $B\in\cB(\X)$.
\end{lemma}

The next lemma with a sufficient condition for aperiodicity is taken from \cite{BCZG23}*{Lemma 3.2}.

\begin{lemma}[Sufficient condition for aperiodicity]\label{prop:aperiodicity}
Assume that $f_\omega$ satisfies \ref{item:H1}. Assume that there exists $\omega_\star\in\Omega_0$ for which $f_\omega$ has a fixed point $\xvec^\star\in\X$, i.e.\ $f_{\omega_\star}(\xvec^\star)=\xvec^\star$. Then for any $A\subset\X$ open containing $\xvec^\star$, there holds that $P(\xvec^\star,A)>0$.
\end{lemma}

\begin{remark}\label{remark:aperiodicity}
    Observe that aperiodicity is always guaranteed for the random ABC flow \eqref{eq:ABCflow-solution} provided that we allow the coefficients $(\A,\B,\C)$ take the trivial value $(0,0,0)$. This choice yields a fixed point not only for the standard one-point process, but also for the projective and two-point processes.
\end{remark}

Irreducibility, small sets and aperiodicity are sufficient conditions for exponential ergodicity of a Markov chain that is defined in a compact space. This will be the case of the one-point $P$ and periodic $\hat{P}$ chains, but not the case of the two-point chain $P^{(2)}$, that presents the problem of the diagonal. In order to deal with the two-point chain we need to address the Lyapunov-drift condition. The following result from \cite{BCZG23}*{Section 4.2} gives a suficient condition for the Lyapunov-drift condition of $P^{(2)}$ to hold true.

\begin{lemma}[Lyapunov-drift condition]\label{lemma:lyapunov-drif-condition}
    Assume that the one-point and the projective chains are geometrically ergodic, and further that the top Lyapunov exponent is positive. There exist $p>0$ sufficiently small, a function $V\in L^1(\T^3\times\T^3)$, $V\geq 1$ with the property
    \[
    \dist_\X(\xvec^1,\xvec^2)^{-p} \lesssim V(\xvec^1,\xvec^2) \lesssim \dist_\X(\xvec^1,\xvec^2)^{-p},
    \]
    and $\varepsilon\in (0,1)$ such that there holds $P^{(2)}V(\xvec^1,\xvec^2) < \varepsilon V(\xvec^1,\xvec^2)$ for all $(\xvec^1,\xvec^2)$ in the complementary of a compact set $K\subsetneq (\T^3\times\T^3)\setminus \Delta$.
\end{lemma}

All in all there is only one missing ingredient to apply this lemma: the positivity of the Lyapunov exponent, that we discuss in the next section.

To prove that the top Lyapunov exponent is positive we need first to make sure that the assumptions for the MET from Section \ref{section:lyapunov} are satisfied. If the one-point chain satisfies Harris theorem, in particular has an ergodic stationary measure $\mu$. Moreover, integrability condition \eqref{eq:MET_integrability} is automatically satisfied by imposing hypothesis \ref{item:H2}.

The following results arises as a sufficieint condition for the top Lyapunov exponent to be positive, and it is derived from Furstenberg's criterion in Proposition \ref{prop:furstenberg}. For simplicity, for this result we will work in the frame of the random ABC vector field specifically, namely $\X=\T^3$.

As for Lemma \ref{prop:small-set}, given $n\in\N$ and $\xvec\in\T^3$, we define $\Phi_{\xvec}:\Omega_0^n\to\T^3$ by $\Phi_{\xvec}(\underline{\omega}^n) = f_{\underline{\omega}^n}(\xvec)$. Additionally we define another mapping $\Psi_{\xvec}:\Omega_0^n\to SL_3(\R)$ as $\Psi_{\xvec}(\underline{\omega}^n) = D_{\xvec}\Phi_{\xvec}$. Notice that $\Psi_{\xvec}$ is $C^1$, and moreover since the flow $f_\omega$ is volume-preserving, then $|\det(D_{\xvec}\Phi_{\xvec})|=|\det(D_{\xvec}f_{\underline{\omega}^n})|=1$ for all $n$, $\omega$ and $\xvec$, and hence $\Psi_{\xvec}(\underline{\omega}^n) \in SL_3(\R)$.

\begin{proposition}\label{prop:lyapunov}
    Under assumptions \ref{item:H1}--\ref{item:H2}, let $f_\omega$ be a continuous RDS that is geometrically ergodic with stationary ergodic measure $\mu$ in $\T^3$. In addition, assume that there exist $n\in\N$, $\underline{\omega}^n_\star\in\Omega_0^n$, and $\xvec^\star\in\X$ that satisfy:
    \begin{enumerate}[label=(L\arabic*), ref=(L\arabic*)]
        \item\label{item:L1} $\exists \varepsilon,\delta>0$ so that $\rho_0(\underline{\omega}^n)\geq \delta$ for all $\underline{\omega}^n\in\Omega_0^n$ with $|\underline{\omega}^n-\underline{\omega}_\star^n|\leq \varepsilon$;
        \item\label{item:L2} $\Phi_{\xvec^\star}(\underline{\omega}^n) = f_{\underline{\omega}^n}(\xvec^\star)$ is a submersion at $\underline{\omega}^n=\underline{\omega}^n_\star$;
        \item\label{item:L3} the restriction of the linear mapping $D_{\underline{\omega}^n_\star}\Psi_{\xvec^\star}$ to $\ker (D_{\underline{\omega}^n_\star}\Phi_{\xvec^\star})$
        \[
        D_{\underline{\omega}^n_\star}\Psi_{\xvec^\star}:  \ker (D_{\underline{\omega}^n_\star}\Phi_{\xvec^\star})  \to  T_{\Psi_{\xvec^\star}(\underline{\omega}^n)}SL_3(\R)
        \]
    is surjective.
    \end{enumerate}
    Then $\lambda_1>0$.
\end{proposition}

The idea behind this result is to impose sufficient restrictions so that the \emph{bad case} for which $\lambda_1=0$ in Proposition \ref{prop:furstenberg} is ruled out. For a discussion about this result see \cite{BCZG23}*{Section 3.2}.

At this stage we have all the tools needed to prove Theorem \ref{thm:mixing}. The to-do list reads a follows.
\begin{itemize}
    \item Check that the basic assumptions \ref{item:H1}--\ref{item:H3} are satisfied.
    \item Check topological irreducibility for the one-point, the projective and the two-point chains. Using Lemma \ref{prop:small-set}, check that the three Markov chains satisfy the small sets condition. Aperiodicity is obtained directly due to Remark \ref{remark:aperiodicity} and Lemma \ref{prop:aperiodicity}. Thanks to Harris Theorem \ref{thm:harris}, this yields geometric ergodicity for the one-point and projective processes since both are defined in compact spaces, $\T^3$ and $\T^3\times\S^2$ respectively.
    \item Check the positivity of the Lyapunov exponent using Proposition \ref{prop:lyapunov}.
    \item Via Lemma \ref{lemma:lyapunov-drif-condition}, geometric ergodicity of the one-point and projective chains together with a positive Lyapunov exponent yield the Lyapunov-drift condition for the two-point chain, concluding via Harris theorem that the two-point chain is $V-$geometrically ergodic.
\end{itemize}

Additionally Theorem \ref{thm:fast-dynamo} will be proved using a nonrandom version of the MET, the positivity of the Lyapunov exponent and the ergodicity of the three Markov chains $P$, $\hat{P}$ and $P^{(2)}$ in Section \ref{section:dynamo}.

\section{Almost-sure exponential mixing: proof of Theorem \ref{thm:mixing}}\label{section:proof-mixing}

In this section we show that the flow defined by the random ABC vector field \eqref{eq:random-abc-vectorfield} satisfies all the requirements to obtain almost sure exponential mixing. As noted in Remark \ref{remark:aperiodicity}, aperiodicty is given by allowing $\A$, $\B$ and $\C$ to take the trivial value zero, so we now put a focus on proving topological irreducibility and small sets of the three Markov chains $P$, $\hat{P}$ and $P^{(2)}$.

\subsection{Topological irreducibility and small sets}

We start by proving irreducibility and small sets. For the first property we will show that the three Markov chains are controllable, by making an explicit choice of parameters that take any initial point to any target point. To prove small sets we use Lemma \ref{prop:small-set}.

\subsubsection{The one-point chain}\label{section:one-point-chain}

Recall that for any $A\in \Bo(\T^3)$ and $\xvec=(x,y,z)\in\T^3$, we define the one-point Markov transition kernel via
\[
P(\xvec,A) = \Prob_0[f_\omega(\xvec) \in A].
\]

\begin{lemma}\label{lemma:control-onepoint-2}
Given any $\xvec,\xvec^\star \in\T^3$, there exist $N\in\N$ and  $\underline{\omega}^N=(\omega_1,\hdots,\omega_N)\in\Omega_0^N$ for which there holds $f_{\underline{\omega}^N}(\xvec)= \xvec^\star$, namely the one-point process is exactly controllable.
\end{lemma}

\begin{proof}
Exact controllability of the one-point chain is indeed achievable in one time step provided that $U\geq \pi$, i.e.\ the coefficients $\A$, $\B$ and $\C$ can take values up to the size of the domain. If this were not the case, then one might to do more steps (a finite number) to reach the target. To see this, notice that one can choose $\alpha$, $\beta$ and $\gamma$ such that $\cos(z+\alpha)=1$, $\cos(x+\beta)=1$ and $\cos(y+\gamma)=1$, so that there yields
\[
x^\star = x + \C, \quad y^\star = y + \A, \quad \text{and} \quad z^\star = z + \B.
\]
Now, since $\dist_\T(x^\star,x),\dist_\T(y^\star,y),\dist_\T(z^\star,z)\leq \pi\leq U$, there is a choice of $\A,\B,\C\in[-U,U]$ for which $\xvec^\star$ is reachable in one step.
\end{proof}

To check that the one-point chain admits a small set we use Lemma \ref{prop:small-set}. Observe that one must only check that the mapping $\Phi_{\xvec}:  \Omega_0^n  \to \T^3$ defined as
\begin{equation}\label{eq:Phimap}
\Phi_{\xvec}(\underline{\omega}^n) = f_{\underline{\omega}^n}(\xvec)
\end{equation}
is a submersion for some $n\in\N$, some $\xvec\in\T^3$ and some $\underline{\omega}^n\in\Omega_0^n$. The remaining assumptions needed to apply Lemma \ref{prop:small-set} are trivially satisfied by the definition of the noise that we are considering in this example.

For the one-point chain notice that for $n=1$, the matrix $D_\omega \Phi_{\xvec}$ is associated to a linear mapping from $\R^6$ to $\R^3$. For instance, we can select 
\[
\xvec=(0,0,0), \quad (\A,\B,\C)=(0,0,0), \quad (\alpha,\beta,\gamma)=\left(\frac{\pi}{2},\frac{\pi}{2},\frac{\pi}{2}\right),
\]
where $\omega=(\A,\B,\C,\alpha,\beta,\gamma)$. Then we obtain the matrix
\[
D_\omega \Phi_{\xvec} =  
\begin{pmatrix}
1 & 0 & 0 & 0 & 0 & 0\\
0 & 1 & 0 & 0 & 0 & 0\\
0 & 0 & 1 & 0 & 0 & 0
\end{pmatrix}
\]
that has full rank. This together with irreducibility and aperiocidity, given that $\T^3$ is a compact domain, yields that $P$ is uniformly geometric ergodic, with the Lebesgue measure as ergodic stationary measure.

\subsubsection{The projective chain}\label{section:projective-chain-2}

Let $\hat{A}\in\Bo(\T^3\times \S^2)$, $\xvec=(x,y,z)\in\T^3$, and $\vvec=(v^x,v^y,v^z)\in\S^2$, we define the projective Markov transition kernel as
\begin{equation}\label{eq:projective-chain}
\hat{P}((\xvec,\vvec),\hat{A})= \Prob_0\left[ \left(f_\omega(\xvec),\frac{D_{\xvec} f_\omega\vvec}{|D_{\xvec} f_\omega\vvec|}\right)\in \hat{A} \right].
\end{equation}
The matrix $D_{\xvec} f_\omega$ is a $3\times 3$ real matrix that has the following entries,
\begin{flalign*}
[D_{\xvec} f_\omega]_{11} & = 1 - \B\C \cos(\beta + x + \A \sin(\alpha + z)) &&\\
& \qquad \times\sin(
   \gamma + y + \A \cos(\alpha + z) + \B \sin(\beta + x + \A \sin(\alpha + z))), &&
\end{flalign*}
\begin{flalign*}
[D_{\xvec} f_\omega]_{12} = -\C \sin(\gamma + y + \A \cos(\alpha + z) + 
   \B \sin(\beta + x + \A \sin(\alpha + z))),  &&
\end{flalign*}
\begin{flalign*}
[D_{\xvec} f_\omega]_{13} & = \A \cos(\alpha + z) - 
 \C (\A\B \cos(\alpha + z) \cos(\beta + x + \A \sin(\alpha + z)) -
    \A \sin(\alpha + z)) &&\\
    & \qquad \times\sin(
   \gamma + y + \A \cos(\alpha + z) + \B \sin(\beta + x + \A \sin(\alpha + z))), &&
\end{flalign*}
\begin{flalign*}
 [D_{\xvec} f_\omega]_{21} = \B \cos(\beta + x + \A \sin(\alpha + z)), &&   
\end{flalign*}
\begin{flalign*}
[D_{\xvec} f_\omega]_{22} = 1, &&
\end{flalign*}
\begin{flalign*}
[D_{\xvec} f_\omega]_{23} = \A\B \cos(\alpha + z) \cos(\beta + x + \A \sin(\alpha + z)) - \A \sin(\alpha + z), && 
\end{flalign*}
\begin{flalign*}
[D_{\xvec} f_\omega]_{31} & = \B \C \cos(\beta + x + \A \sin(\alpha + z)) &&\\
    & \qquad \times\cos(
   \gamma + y + \A \cos(\alpha + z) + 
    \B \sin(\beta + x + \A \sin(\alpha + z))) &&\\
    & \qquad - 
 \B \sin(\beta + x + \A \sin(\alpha + z)), &&
\end{flalign*}
\begin{flalign*}
[D_{\xvec} f_\omega]_{32} = \C \cos(\gamma + y + \A \cos(\alpha + z) + 
   \B \sin(\beta + x + \A \sin(\alpha + z))), &&  
\end{flalign*}
\begin{flalign*}
[D_{\xvec} f_\omega]_{33} & = 1 + \C \cos(
   \gamma + y + \A \cos(\alpha + z) + 
    \B \sin(\beta + x + \A \sin(\alpha + z))) &&\\
    & \qquad \times(\A\B \cos(\alpha + z) \cos(
      \beta + x + \A \sin(\alpha + z)) - \A \sin(\alpha + z)) &&\\
    & \qquad - 
 \A\B \cos(\alpha + z) \sin(\beta + x + \A \sin(\alpha + z)). &&
\end{flalign*}
Even if componentwise this matrix seems certainly cluttered, there is some hidden structure. Let us introduce the following variables,
\begin{equation}\label{eq:new-variables}
Z = \alpha + z, \quad X=\beta+x+\A\sin Z, \quad Y = \gamma + y + \A\cos Z + \B\sin X,
\end{equation}
then we can write the matrix $D_{\xvec} f_\omega$ as
\[
D_{\xvec} f_\omega =  
\begin{pmatrix}
1-\B\C\cos X\sin Y & -\C\sin Y & -\A(\C\Phi\sin Y-\cos Z)\\
\B\cos X  & 1 & \A\Phi \\
\B(\C\cos X\cos Y-\sin X)  & \C\cos Y & 1+\A(\C\Phi\cos Y-\B\cos Z\sin X) 
\end{pmatrix},
\]
where $\Phi=\B\cos X\cos Z-\sin Z$. This matrix has the property that
\[
\det D_{\xvec} f_\omega = 1
\]
for all $\xvec\in\T^3$ and $\omega\in\Omega_0$. Invertibility of $D_{\xvec} f_\omega$ is a key feature that we will use to prove irreducibility of the projective chain. For this reason it is convenient to display the form that its inverse matrix has. Using the short hand notation \eqref{eq:new-variables}, we find that
\[
(D_{\xvec} f_\omega)^{-1} =  
\begin{pmatrix}
1-\A\B\sin X\cos Z & -\C(\A\Psi \cos Z-\sin Y) & -\A\cos Z\\
\B(\A\sin X\sin Z-\cos X)  & 1+\C(\A\Psi\sin Z-\B\cos X\sin Y) & \A\sin Z \\
\B\sin X  & \C\Psi & 1
\end{pmatrix},
\]

To show topological irreducibility of the projective Markov chain we will prove that, as the one-point chain, it is exactly controllable.

\begin{proposition}\label{prop:control-projective-2}
Given any two pairs $(\xvec,\vvec),(\xvec^\star,\vvec^\star)\in \T^3\times\S^2$, there exist $N\in\N$ and $\underline{\omega}^N=(\omega_1,\hdots,\omega_N)\in \Omega_0^N$ such that
\[
f_{\underline{\omega}^N}(\xvec) = \xvec^\star \quad \text{and} \quad \frac{D_{\xvec} f_{\underline{\omega}^N}\vvec}{|D_{\xvec} f_{\underline{\omega}^N}\vvec|}= \vvec^\star.
\]
\end{proposition}

The proof of this proposition can be split into three different parts. First of all, we prove that the vector $\overline{\vvec} = (1,0,0)$ is reachable from any $\vvec\in\S^2$. Second, we prove that we can get to any given $\overline{\xvec}\in\T^3$ from any point $\xvec\in\T^3$ only via \emph{rigid motions}, namely keeping $\overline{\vvec} = (1,0,0)$ fixed. Last, we show that we can go exactly from $\overline{\vvec} = (1,0,0)$ to the target $\vvec^\star$ in such a way that we end up exactly the target point $\xvec^\star$ for the dynamical system in $\T^3$.

Let us introduce the short notation defined recursively via
\[
\xvec_n = f_{\omega_n}(\xvec_{n-1})\in\T^3, \quad \vvec_n = \frac{D_{\xvec_{n-1}}f_{\underline{\omega}^n}\vvec_{n-1}}{|D_{\xvec_{n-1}}f_{\underline{\omega}^n}\vvec_{n-1}|}\in\S^2,
\]
with $\xvec_0=\xvec$ and $\vvec_0=\vvec$.  The first lemma we want to introduce displays the fact the the second component of the projective process, denoted by $\vvec_n$, is also exactly controllable.

\begin{lemma}\label{lemma:correct-angle-projective-2}
Let $\vvec\in\S^2$. The following results hold true.
\begin{enumerate}[label=(\arabic*), ref=(\arabic*)]
\item\label{item:control1} There exist $N_1\in\N$ and $\underline{\omega}^{N_1}=(\omega_1,\hdots,\omega_{N_1})\in\Omega_0^{N_1}$ such that $\vvec_{N_1} = (1,0,0)$. 
\item\label{item:control2} If initially $\vvec = (1,0,0)$, then given any $\overline{\vvec}\in\S^2$, there exist $N_2\in\N$ and $\underline{\omega}^{N_2}=(\omega_1,\hdots,\omega_{N_2})\in\Omega_0^{N_2}$ such that $\vvec_{N_2}=\overline{\vvec}$.
\end{enumerate}
\end{lemma}

\begin{proof}
The first point that we will show is that the process $\vvec_n$ can actually reach one of the poles in one time step. Since in the claim we are choosing specifically the pole $(1,0,0)$, we will show that anyways one can reach this pole in at most two time steps from any starting point $
\vvec\in\S^2$. Fix $\vvec\in\S^2$ and $\xvec\in\T^3$, then notice that we look for a choice of $\omega\in\Omega_0$ and some $k\in\R$, $k\neq 0$, so that 
\[
\vvec =  
\begin{pmatrix}
v^x \\
v^y \\
v^z
\end{pmatrix}
 = (D_{\xvec} f_\omega)^{-1}
\begin{pmatrix}
k \\
0 \\
0
\end{pmatrix} = k 
\begin{pmatrix}
1-\A\B\sin X\cos Z \\
\B(\A\sin X\sin Z-\cos X) \\
\B\sin X
\end{pmatrix},
\]
where we use the variables introduced in \eqref{eq:new-variables}. We can choose $\alpha\in [0,2\pi)$ so that $\cos(\alpha+z)=1$, and hence
\[
\begin{pmatrix}
v^x \\
v^y \\
v^z
\end{pmatrix}= k 
\begin{pmatrix}
1-\A\B\sin (\beta+x)) \\
-\B\cos (\beta+x) \\
\B\sin (\beta+x)
\end{pmatrix}.
\]
Now, if $v^z\neq 0$, choose $k=1$, $\B\neq 0$ and $\beta\in [0,2\pi)$ so that $v^y=\B\cos(\beta+x)$ and $v^z=\B\sin(\beta+x)$, and choose $\A$ such that $v^x=1-\A v^z$. However, if $v^z=0$ and $v^x\neq 0$, choose $\beta$ so that $\sin(\beta+x)=0$, and $\B$ so that the vector $(1,-\B,0)$ is parallel to $(v^x,v^y,0)$. 

It is clear that the only situation that is not achievable in one time step and using this choice of parameters is $v^x=v^z=0$. Arguably this is not a problem because this would be one of the degenerate situations that we are looking for. However, for the sake of completion we will show that one can also take the vector $(0,1,0)$ to $(1,0,0)$ in two time steps. To see this, notice that
\[
D_{\xvec} f_\omega 
\begin{pmatrix}
0 \\
1 \\
0
\end{pmatrix} =  
\begin{pmatrix}
-\C \sin(\gamma + y + \A \cos(\alpha + z) + 
   \B \sin(\beta + x + \A \sin(\alpha + z))) \\
1 \\
\C \cos(\gamma + y + \A \cos(\alpha + z) + 
   \B \sin(\beta + x + \A \sin(\alpha + z)))
\end{pmatrix},
\]
hence it is clear that one can choose $\omega\in \Omega_0$ so that the first component is not vanishing. Now, repeating the previous argument for the case $v^x\neq 0$ and there yields the claim of \ref{item:control1}.

To prove \ref{item:control2}, we argue in a similar fashion. First of all, notice that
\begin{equation}\label{eq:DF100}
D_{\xvec} f_\omega 
\begin{pmatrix}
1 \\
0 \\
0
\end{pmatrix} =  
\begin{pmatrix}
1-\B\C\cos X\sin Y\\
\B\cos X  \\
\B(\C\cos X\cos Y-\sin X)  
\end{pmatrix}.
\end{equation}
In order to simplify the computations, let us choose as before $\alpha\in [0,2\pi)$ so that $\cos(\alpha+z)=1$,
\[
D_{\xvec} f_\omega 
\begin{pmatrix}
1 \\
0 \\
0
\end{pmatrix} =  
\begin{pmatrix}
1 - \B\C \cos(\beta + x) \sin(\A + \gamma + y + \B \sin(\beta + x)) \\
\B\cos (\beta+x) \\
\B\C \cos(\beta + x) \cos(\A + \gamma + y + \B \sin(\beta + x)) - 
 \B \sin(\beta + x)
\end{pmatrix}.
\]
The goal is to reach the target vector $\vvec=(v^x,v^y,v^z)$, however this might not be possible in one time step since $\A,\B,\C\in[-U,U]$. The strategy will be to take an intermediate step in the vector $\overline{\vvec}=(1,1,1)/\sqrt{3}$. Let us choose $\beta\in [0,2\pi)$ so that $\cos(\beta+x)=1$, hence we look for parameters $\A$, $\B$, $\gamma$ and $k$ such that
\[
\begin{pmatrix}
1 - \B\C \sin(\A + \gamma + y) \\
\B \\
\B\C \cos(\A + \gamma + y)
\end{pmatrix} = k 
\begin{pmatrix}
1 \\
1 \\
1
\end{pmatrix}.
\]
This is an admissible step since, for instance, $k=1$, $\B=1$, $\C=1$ and $\A+\gamma$ such that $\cos(\A+\gamma+y)=1$, yields a solution. 

On a second time step we want to reach the target vector $\vvec=(v^x,v^y,v^z)$ from $\overline{\vvec}=(1,1,1)/\sqrt{3}$. Now it is more convenient to consider the inverse map from $\vvec$. To simplify the computations as before, we choose $\alpha$ and $\beta$ so that $\cos(\alpha+z)=1$ and $\cos(\beta+x)=1$, and in addition now we impose $\A=0$, then we obtain
\[
(D_{\xvec} f_\omega)^{-1} 
\begin{pmatrix}
v^x \\
v^y \\
v^z
\end{pmatrix} =  
\begin{pmatrix}
v^x + \C  \sin(\gamma + y)v^y\\
-\B v^x +  (1 - \B\C \sin(\gamma + y))v^y \\
v^z - \C \cos(\gamma + y)v^y
\end{pmatrix}.
\]
Let us assume for the moment that $v^y\neq 0$, and notice that for the mapping
\[
\wvec =  
\begin{pmatrix}
w^x \\
w^y \\
w^z
\end{pmatrix}
= (D_{\xvec} f_\omega)^{-1}
\begin{pmatrix}
v^x \\
v^y \\
v^z
\end{pmatrix} = 
\begin{pmatrix}
v^x + \C  \sin(\gamma + y)v^y\\
-\B v^x +  (1 - \B\C \sin(\gamma + y))v^y \\
v^z - \C \cos(\gamma + y)v^y
\end{pmatrix}
\]
the second component has a fixed point $w^y=v^y$ provided that $\B=0$. Therefore, for fixed $v^y\neq 0$, there exist some $N\in\N$ and finite sequences $\{\C_n\}_{1\leq n\leq N}\subset [-U,U]$, $\{\gamma_n\}_{1\leq n\leq N}\subset [0,2\pi)$, defined by
\[
\wvec_n =  
\begin{pmatrix}
w^x_n \\
w^y_n \\
w^z_n
\end{pmatrix} = 
\begin{pmatrix}
w^x_{n-1} + \C_n  \sin(\gamma_n + y)v^y \\
v^y \\
w^z_{n-1} - \C_n \cos(\gamma_n + y)v^y
\end{pmatrix},
\]
with $\wvec_0=\vvec$ and such that $w^x_N=w^z_N=k$ for some $k\in\R$, $k\neq 0$. Notice that the coordinate $y$ in these dynamics is a fixed point, it should in principle change with $n$, but it does not since $\A_n=\B_n=0$ for all $1\leq n\leq N$. In the next step, the choice $\C_{N+1}=0$ and $\B_{N+1}=-1$ yields that
\[
\wvec_{N+1} = k 
\begin{pmatrix}
1 \\
1 \\
1
\end{pmatrix}
\]
which concludes the proof of the lemma for all target points $\vvec$ with second component $v^y\neq 0$. For the remaining case just notice that we can take a preliminary step where now we do not assume $\A=0$ as before and define
\[
\wvec_0 = (D_{\xvec} f_\omega)^{-1} 
\begin{pmatrix}
v^x \\
0 \\
v^z
\end{pmatrix}
= 
\begin{pmatrix}
v^x - \A v^z \\
-\B v^x \\
v^z
\end{pmatrix}.
\]
If $v^x\neq 0$, any choice of $\B\neq 0$ produces $w^y\neq 0$. If we are in the degenerate situation that $v^x=v^z=0$ then $v^z=1$, so selecting $\A\neq 0$ one produces $w^x\neq 0$. Finally, repeating the scheme designed for the case $v^y\neq 0$ we find that the proof of the lemma is complete.
\end{proof}

\begin{lemma}\label{lemma:rigid-motion-projective-2}
Assume that $\vvec = (1,0,0)$. Then given any $\xvec,\overline{\xvec}\in\T^2$ and any $\varepsilon>0$, there exist $N_3\in\N$ and $\underline{\omega}^{N_3}=(\omega_1,\hdots,\omega_{N_3})\in\Omega_0^{N_3}$ such that
\[
f_{\underline{\omega}^{N_3}}(\xvec)=\overline{\xvec}, \quad \text{and}\quad \vvec_{N_3}= (1,0,0).
\]
\end{lemma}

\begin{proof}
On the one hand, notice that if one wants keep the vector $\vvec=(1,0,0)$ a fixed point in the projective dynamics, it suffices to fix the parameter $\B=0$, see \eqref{eq:DF100}.
On the other hand, as it is evident from the construction of the one-point chain, even with the restriction $\B=0$ the process is exactly controllable. To picture this, notice that if $\B=0$ then
\[
\begin{pmatrix}
x_1 \\
y_1 \\
z_1
\end{pmatrix}
 = 
\begin{pmatrix}
x + \C \cos(\gamma + y + \A \cos(\alpha + z)) + \A \sin(\alpha + z) \\
y + \A \cos(\alpha + z) \\
z + \C \sin(\gamma + y + \A \cos(\alpha + z))
\end{pmatrix}.
\]
Now choose $\alpha$ so that $\cos(\alpha + z)=1$ and $\A$ such that $y+\A=\overline{y}$. Then there yields,
\[
\begin{pmatrix}
x_1 \\
y_1 \\
z_1
\end{pmatrix}
 = 
\begin{pmatrix}
x + \C \cos(\gamma + \overline{y}) \\
\overline{y} \\
z + \C \sin(\gamma + \overline{y})
\end{pmatrix}
.
\]
The goal now is to reach the point $(\overline{x},\overline{z})$ from $(x,z)$ in the plane $\{y=\overline{y}\}$. It is possible that such thing is not doable in one time step since $|\C|\leq U$, but it is clear that $(\overline{x},\overline{z})$ is reachable in a finite number of steps. In particular, the maximum distance between $(\overline{x},\overline{z})$ and $(x,z)$ is $\sqrt{2}\pi$, then if we assume $U=\pi$, this means that with these constrains any point can be reached at most in two time steps.
\end{proof}

\begin{proof}[Proof of Proposition \ref{prop:control-projective-2}]
Let $\varepsilon>0$, and fix $(\xvec,\vvec),(\xvec^\star,\vvec^\star)\in\T^3\times\S^2$. Using Lemma \ref{lemma:correct-angle-projective-2}, there exist $N_1>0$ and $\underline{\omega}^{N_1}\in\Omega_0^{N_1}$ such that 
\[
\vvec_{N_1} (\vvec)= \frac{D_{\xvec_{N_1}}f_{\underline{\omega}^{N_1}}\vvec}{|D_{\xvec_{N_1}}f_{\underline{\omega}^{N_1}}\vvec|} = \overline{\vvec} = 
\begin{pmatrix}
1\\
0\\
0
\end{pmatrix}
.
\]
This process brings the point $\xvec\in\T^3$ to $\xvec_{N_1}\in\T^3$. An application of the same lemma yields that there exist $N_2>0$ and $\underline{\omega}^{N_2}\in\Omega_0^{N_2}$ so that
\[
\vvec_{N_2}(\overline{\vvec}) = \frac{D_{\xvec_{N_2}}f_{\underline{\omega}^{N_2}}\overline{\vvec}}{|D_{\xvec_{N_2}}f_{\underline{\omega}^{N_2}}\overline{\vvec}|} = \vvec^\star,
\]
hence we define
\[
\overline{\xvec} = f_{\underline{\omega}^{N_2}}^{-1}(\xvec^\star).
\]
Lemma \ref{lemma:rigid-motion-projective-2} implies that there exist $N_3\in\N$ and $\underline{\omega}^{N_3}\in\Omega_0^{N_3}$ such that
\[
f_{\underline{\omega}^{N_3}}(\xvec_{N_1})=\overline{\xvec} \quad \text{and}\quad \vvec_{N_3}= 
\begin{pmatrix}
1\\
0\\
0
\end{pmatrix}
.
\]
All in all, we set the noise vector $\underline{\omega}^N = (\underline{\omega}^{N_1},\underline{\omega}^{N_3},\underline{\omega}^{N_2})\in\Omega_0^N$, where $N=N_1+N_3+N_2$, so that we obtain $\vvec_N=\vvec^\star$ and $f_{\underline{\omega}^N}(\xvec)=\xvec^\star$, and thus there yields the claim of the proposition.
\end{proof}

\begin{remark}[On the invertibility of the flow map]\label{remark:inverse-flow-map-2}
In the proof of Proposition \ref{prop:control-projective-2} we have used the inverse map $f_{\underline{\omega}^{N_2}}^{-1}(\xvec^\star)$. Give a noise realisation $\omega\in\Omega_0$, the inverse flow map is defined by
\[
f^{-1}_\omega(\xvec) = \left(f_{(\A,\alpha)}^{-1}\circ f_{(\B,\beta)}^{-1}\circ f_{(\C,\gamma)}^{-1}\right)(\xvec),
\]
where the inverse of each of the partial applications of the flow can be presented from two different perspectives,
\begin{enumerate}
    \item either we can choose to define
    \[
    f_{(\A,\alpha)}^{-1} = f_{(\A,\alpha+\pi)}, \quad f_{(\B,\beta)}^{-1} = f_{(\B,\beta+\pi)}, \quad f_{(\C,\gamma)}^{-1} = f_{(\C,\gamma+\pi)};
    \]
    \item or either
    \[
    f_{(\A,\alpha)}^{-1} = f_{(-\A,\alpha)}, \quad f_{(\B,\beta)}^{-1} = f_{(-\B,\beta)}, \quad f_{(\C,\gamma)}^{-1} = f_{(-\C,\gamma)}.
    \]
\end{enumerate}
Both definitions yield that $f_{\omega}^{-1}(f_{\omega}(\xvec)) = \xvec$, and moreover they fall under the theoretical frame here considered with small modifications. Therefore one can apply Lemma \ref{lemma:rigid-motion-projective-2} as well to the inverse flow map as required in the proof of the proposition.
\end{remark}

To show that the projective process admits a small set we find $\xvec$ and $\underline{\omega}^n$ for which the assumption of Lemma \ref{prop:small-set} are satisfied.

Given $n\in\N$ and $(\xvec,\vvec)\in\T^3\times\S^2$ we define the mapping
$\hat{\Phi}_{\xvec,\vvec}:  \Omega_0^n  \to  \T^3\times\S^2$ as
\begin{equation}
\hat{\Phi}_{\xvec,\vvec}(\underline{\omega}^n) = \left(f_{\underline{\omega}^n}(\xvec),\frac{D_{\xvec} f_{\underline{\omega}^n}\vvec}{|D_{\xvec} f_{\underline{\omega}^n}\vvec|}\right).
\end{equation}
It is sufficient to consider the case $n=1$, so $D_{\omega}\hat{\Phi}_{(\xvec,\vvec)}$ can be understood a linear map from $\R^6$ to $\R^6$. The sufficient condition for the projective process to admit a small set then boils down to find a collection of $\xvec$, $\vvec$ and $\omega=(\A,\B,\C,\alpha,\beta,\gamma)$ for which $\text{rank} (D_{\omega}\hat{\Phi}_{(\xvec,\vvec)})\geq 5$. In this regard let us define
\[
\xvec=(0,0,0),\quad \vvec=(0,1,0),\quad (\A,\B,\C)=\left( \pi,0,1 \right), \quad (\alpha,\beta,\gamma)=(0,0,0),
\]
which yields
\[
D_{\omega}\hat{\Phi}_{(\xvec,\vvec)} = 
\begin{pmatrix}
0 & 0 & -1 & \pi & 0 & 0\\
1 & 0 & 0 & 0 & 0 & 0\\
-1 & 1 & 0 & 0 & 0 & -1\\
1/\sqrt{2} & 0 & 0 & 0 & 0 & 1/\sqrt{2}\\
0 & 0 & -1/(2\sqrt{2}) & 0 & 0 & 0\\
0 & 0 & -1/(2\sqrt{2}) & 0 & 0 & 0
\end{pmatrix}.
\]
This matrix has rank $5$ since the first five rows are linearly independent, thus from Lemma \ref{prop:small-set} it follows that the projective chain admits a small set. This result together with irreducibility and aperiodicity yields in fact that $\hat{P}$ is uniformly geometrically ergodic in $\T^3\times\S^2$.

\subsubsection{The two-point chain}\label{sectio:two-point-chain}

The two-point chain is arguably the most important object in the study of mixing. It gives information about the position of two particles starting at $\xvec^1=(x^1,y^1,z^1)\in\T^3$ and $\xvec^2=(x^2,y^2,z^2)\in\T^3$ respectively, with $\xvec^1\neq\xvec^2$. After $n\in\N$ iterations in which both particles are advected by the same vector field with the same noise realisation $\underline{\omega}\in\Omega$, the position of the particles is given by $f^n_{\underline{\omega}}(\xvec^1)$ and $f^n_{\underline{\omega}}(\xvec^2)$ respectively, where there holds as well $f^n_{\underline{\omega}}(\xvec^1)\neq f^n_{\underline{\omega}}(\xvec^2)$.

Precisely to hold account of such points in the diagonal we introduce the notation 
\[
\Delta = \{(\xvec^1,\xvec^2)\in\T^3\times\T^3\mid \xvec^1=\xvec^2\}\subset \T^3\times\T^3.
\]
The necessity of removing this subset comes from the fact that there is a invariant measure supported in $\Delta$, namely the complement set of the diagonal is not reachable from the diagonal. Therefore exponential ergodicity in the full space $\T^3\times\T^3$ is not viable. In particular, the dynamics in $\Delta$ are perfectly described by the one-point chain and well-posedness of the ODE
\[
\frac{\dd}{\dd t}\phi_t(\xvec) = u(t,\phi_t(\xvec)), \quad t>0,\xvec\in\T^3
\]
with a Lipschitz vector field $u$ like \eqref{eq:random-abc-vectorfield} yields that the neither the diagonal is reachable from out of the diagonal, giving really (at least) two different invariant measures in $\T^3\times\T^3$.

With this notation in hand, for any Borel set $A^{(2)}\in\cB((\T^3\times\T^3)\setminus\Delta)$, and any $(\xvec^1,\xvec^2)\in(\T^3\times\T^3)\setminus\Delta$, we define the two-point Markov chain via
\begin{equation}\label{eq:twopoint-chain-2}
P^{(2)}((\xvec^1,\xvec^2),A^{(2)})= \Prob_0\left[ \left(f_\omega(\xvec^1),f_\omega(\xvec^2)\right)\in A^{(2)} \right].
\end{equation}
Just like for the projective chain, let us introduce the following short notation,
\[
\xvec^1_n = f_{\omega_n}(\xvec^1_{n-1}) \quad \text{and}\quad \xvec^2_n = f_{\omega_n}(\xvec^2_{n-1})
\]
where $\xvec^1_0 = \xvec^1$ and $\xvec^2_0 = \xvec^2$. Analogously to the previous section, to prove that that the two-point chain is topologically irreducible, we will show that it is indeed exactly controllable.

\begin{proposition}\label{prop:control-twopoint-2}
Given any two elements $(\xvec^1,\xvec^2),(\xvec^{\star,1},\xvec^{\star,2})\in(\T^3\times\T^3)\setminus\Delta$, there exist $N\in\N$ and $\underline{\omega}^N=(\omega_1,\hdots,\omega_N)\in \Omega_0^N$ such that
\[
f_{\underline{\omega}^N}(\xvec^1)=\xvec^{\star,1} \quad \text{and} \quad f_{\underline{\omega}^N}(\xvec^2)=\xvec^{\star,2}.
\]
\end{proposition}

The proof of this proposition is split in a number of steps. The strategy consists on selecting appropriate noise realisations so that we can bring the two particles arbitrarily close, then we can bring one of them to a target $\overline{\xvec}$ via rigid rotations, so that the distance between both particles remains fixed. Finally, we close the argument using the invertibility of the vector field.

\begin{lemma}\label{lemma:rigid-motion-twopoint1-2}
Let $(\xvec^1,\xvec^2)\in(\T^3\times\T^3)\setminus\Delta$. Given any $\varepsilon>0$, there exists $N_1\in\N$ and $\underline{\omega}^{N_1}\in\Omega_0^{N_1}$ such that $x^1_{N_1}=x^2_{N_1}$, $y^1_{N_1}=y^2_{N_1}$, and $|z^1_{N_1}-z^2_{N_1}|<\varepsilon$. In particular,
\[
\dist_{\T^3}(\xvec^1_{N_1},\xvec^2_{N_1}) <\varepsilon.
\]
\end{lemma}

\begin{proof}
Let $\varepsilon>0$ be fixed, and assume that $|z_0^1-z_0^2|\geq \varepsilon$. First of all let us choose $\B=\C=0$, then notice that we can write the distance between the two points as
\begin{align*}
x_1^1-x_1^2 & =  x_0^1-x_0^2  +  \A(\sin(z_0^1+\alpha)-\sin(z_0^1+\alpha)) \\ 
 & =  x_0^1-x_0^2  +  2\A\sin\left(\frac{z^1_0-z^2_0}{2}\right)\cos\left(\frac{z^1_0+z^2_0}{2} + \alpha\right),\\
y_1^1-y_1^2 & =  y_0^1-y_0^2  +  \A(\cos(z_0^1+\alpha)-\cos(z_0^1+\alpha)) \\ 
  & =  y_0^1-y_0^2  +   2\A\sin\left(\frac{z^1_0-z^2_0}{2}\right)\sin\left(\frac{z^1_0+z^2_0}{2} + \alpha\right),
\end{align*}
and $z_0^1-z_0^2=z_1^1-z_1^2$. Now we select $\alpha$ such that
\[
\cos\left(\frac{z^1_0+z^2_0}{2} + \alpha\right) = 1,
\]
namely, $x_0^1-x_0^2=x_1^1-x_1^2$ is a fixed point as well. The idea is to choose $\A$ such that $y_0^1-y_0^2=\varepsilon$, however this might not be possible in one iteration because $|A|\leq U$. Nonetheless, we can define a finite sequence $\{\A_n\}_{1\leq n\leq M_1}\subset [-U,U]$ for some $M_1\in\N$, and
\[
x_n^1-x_n^2 = x_{n-1}^1-x_{n-1}^2 + 2\A_n\sin\left(\frac{z^1_0-z^2_0}{2}\right),
\]
such that there holds $x_{M_1}^1-x_{M_1}^2=\varepsilon$. In a similar fashion, we find that there exists a some natural number $M_2$, so that choosing now $\alpha_n$ such that
\[
\cos\left(\frac{z^1_0+z^2_0}{2} + \alpha_n\right) = 0,
\]
for all $M_1+1\leq n\leq M_2$, we find a find another finite sequence $\{\A_n\}_{M_1+1\leq n\leq M_2}\subset [-U,U]$ such that $y_{M_2}^1-y_{M_2}^2=\varepsilon$ and $x_{M_2}^1-x_{M_2}^2=\varepsilon$. Furthermore, by now choosing $\A=\B=0$ and $\C_n$ and $\gamma_n$ appropriately, we can repeat this argument to bring $z^1_n+z^2_n$ as well to distance $\varepsilon$ in $M_3-M_2-M_1\in\N$ iterations, so that
\[
x_{M_3}^1-x_{M_3}^2=y_{M_3}^1-y_{M_3}^2=z_{M_3}^1-z_{M_3}^2=\varepsilon.
\]
Finally, selecting again $\B_{M_3+1}=\B_{M_3+1}=0$, and $\alpha_{M_3+1}$ so that $y_{M_3+1}^1-y_{M_3+1}=y_{M_3}^1-y_{M_3}$, we find that we can choose $\A_{M_3+1}$ such that
\[
0 = x_{M_3+1}^1-x_{M_3+1} = \varepsilon + 2\A_{M_3+1}\sin\left(\frac{\varepsilon}{2}\right),
\]
since
\[
\A_{M_3+1} = -\frac{\varepsilon}{2\sin(\varepsilon/2)} \in[-U,U] 
\]
provided that $U>1$ and $\varepsilon$ is small enough. In particular, the choice $U=\pi$ yields an upper bound for the admissible $\varepsilon>0$. To sum up, we apply one final iteration with $\alpha_{M_3+2} = \alpha_{M_3+1}+\pi/2$ and the same choice of $\A_{M_3+2}=\A_{M_3+1}$, so that $y_{M_3+2}^1-y^2_{M_3+2}=0$, and thus there concludes the proof of the lemma.
\end{proof}

\begin{lemma}\label{lemma:rigid-motion-twopoint2-2}
Let $(\xvec^1,\xvec^2)\in(\T^3\times\T^3)\setminus\Delta$ be such that $x^1-x^2=0$ and $y^1-y^2=0$. Given any $\varepsilon>0$ and any $\overline{\xvec}\in\T^3$, there exist $N_2\in\N$ and $\underline{\omega}^{N_2}\in\Omega_0^{N_2}$ so that 
\[
\dist_{\T^3}(\xvec^1_{N_2},\overline{\xvec}) <\varepsilon\qquad \text{and} \qquad\dist_{\T^3}(\overline{\xvec},\xvec^2_{N_2})=\dist_{\T^3}(\xvec^1,\xvec^2).
\]
\end{lemma}

\begin{proof}
The proof of this lemma mimics the proof of Lemma \ref{lemma:rigid-motion-projective-2}. On the one hand, notice that if $x^1-x^2=0$ and $y^1-y^2=0$, then it suffices to impose $\A_n=0$ in order to make sure that $\xvec^1_n-\xvec^2_n=\xvec^1-\xvec^2$ for all $n$. Indeed,
\begin{align}
[f_{(\A,\alpha)}(\xvec^1)]_1-[f_{(\A,\alpha)}(\xvec^2)]_1 &= x^1-x^2 + 2\A\sin\left(\frac{z^1-z^2}{2}\right)\cos\left(\frac{z^1+z^2}{2} + \alpha\right),\\
[f_{(\A,\alpha)}(\xvec^1)]_2-[f_{(\A,\alpha)}(\xvec^2)]_2 &= y^1-y^2 + 2\A\sin\left(\frac{z^1-z^2}{2}\right)\sin\left(\frac{z^1+z^2}{2} + \alpha\right),\\
[f_{(\A,\alpha)}(\xvec^1)]_3-[f_{(\A,\alpha)}(\xvec^2)]_3 &= z^1-z^2,
\end{align}
so $\A=0$ imply $f_{(\A,\alpha)}(\xvec^1)-f_{(\A,\alpha)}(\xvec^2) = \xvec^1-\xvec^2$. Analogous arguments with the $(\B,\beta)-$flow and the $(\C,\gamma)-$flow but using the fact that $x^1-x^2=0$ and $y^1-y^2=0$ yield that
\[
f_\omega(\xvec^1)-f_\omega(\xvec^2) = \left( f_{(\C,\gamma)}\circ f_{(\B,\beta)}\circ f_{(\A,\alpha)} \right)(\xvec^1) - \left( f_{(\C,\gamma)}\circ f_{(\B,\beta)}\circ f_{(\A,\alpha)} \right)(\xvec^2) = \xvec^1 - \xvec^2.
\]
Now, as seen in Lemma \ref{lemma:rigid-motion-projective-2}, the one-point chain is exactly controllable with one of the intensities being 0, in this case $\A=0$, and this concludes the proof of the lemma.
\end{proof}

\begin{proof}[Proof of Proposition \ref{prop:control-twopoint-2}]
Fix $(\xvec^1,\xvec^2),(\xvec^{\star,1},\xvec^{\star,2})\in(\T^3\times\T^3)\setminus\Delta$ and fix $\varepsilon>0$. On the one hand we can apply Lemma \ref{lemma:rigid-motion-twopoint1-2} to the inverse of the flow map $f^{-1}_\omega$ (see Remark \ref{remark:inverse-flow-map-2}), to deduce that there exist $N_1\in\N$ and $\underline{\omega}^{N_1}\in\Omega_0^{N_1}$ such that
\[
\overline{x}^1 = \overline{x}^2, \quad \overline{y}^1 = \overline{y}^2 \quad \text{and} \quad \dist_{\T^3}(\overline{z}^1,\overline{z}^2)=\varepsilon_0,
\]
where $\overline{\xvec}^1 = f_{\underline{\omega}^{N_1}}^{-1}(\xvec^{\star,1})$ and $\overline{\xvec}^2 = f_{\underline{\omega}^{N_1}}^{-1}(\xvec^{\star,2})$. On the other hand, Lemma \ref{lemma:rigid-motion-twopoint1-2} applied to the (forward) flow map starting from $(\xvec^1,\xvec^2)$ yields that we can find $N_2\in\N$ and $\underline{\omega}^{N_2}\in\Omega_0^{N_2}$ such that
\[
[f_{\underline{\omega}^{N_2}}(\xvec^1)]_1 = [f_{\underline{\omega}^{N_2}}(\xvec^2)]_1, \quad [f_{\underline{\omega}^{N_2}}(\xvec^1)]_2 = [f_{\underline{\omega}^{N_2}}(\xvec^2)]_2,
\]
\[
[f_{\underline{\omega}^{N_2}}(\xvec^1)]_3 - [f_{\underline{\omega}^{N_2}}(\xvec^2)]_3=\varepsilon.
\]
Now, Lemma \ref{lemma:rigid-motion-twopoint2-2} ensures that there exist $N_3\in\N$ and $\underline{\omega}^{N_2}\in\Omega_0^{N_3}$ such that
\[
f_{\underline{\omega}^{N_3}}( f_{\underline{\omega}^{N_2}}(\xvec^1)) = \overline{\xvec}^1 \quad \text{and} \quad f_{\underline{\omega}^{N_3}}( f_{\underline{\omega}^{N_2}}(\xvec^1) )- f_{\underline{\omega}^{N_3}}( f_{\underline{\omega}^{N_2}}(\xvec^2) = 
\begin{pmatrix}
0 \\
0 \\
\varepsilon
\end{pmatrix}
.
\]
To sum up, we define $N=N_2+N_3+N_1$ and the noise vector $\underline{\omega}^N=(\underline{\omega}^{N_2},\underline{\omega}^{N_3},\underline{\omega}^{N_1})\in\Omega_0^N$, then we found that
\[
f_{\underline{\omega}^{N}}( \xvec^1 )= \xvec^{\star,1} \quad \text{and} \quad f_{\underline{\omega}^{N}}( \xvec^2 )= \xvec^{\star,2},
\]
and the statement of the proposition follows.
\end{proof}

Regarding the existence of a small set for the two point process, we go back to Lemma \ref{prop:small-set} yet again. Given any $n\in\N$ and $(\xvec^1,\xvec^2)\in(\T^3\times\T^3)\setminus\Delta$, we define the mapping
\begin{align*}
\Phi^{(2)}_{(\xvec^1,\xvec^2)} :  \Omega_0^n & \to  (\T^3\times\T^3)\setminus\Delta \\
  \underline{\omega}^n \ &\mapsto  \ \Phi^{(2)}_{(\xvec^1,\xvec^2)} (\underline{\omega}^n) = (f_{\underline{\omega}^n}(\xvec^1),f_{\underline{\omega}^n}(\xvec^2)).
\end{align*}
Hence, by Lemma \ref{prop:small-set}, as long as we can find a choice of $n\in\N$, $(\xvec^1,\xvec^2)\in(\T^3\times\T^3)\setminus\Delta$ and $\underline{\omega}^n\in\Omega_0^n$ for which $D_{\underline{\omega}^n}\Phi^{(2)}_{(\xvec^1,\xvec^2)} $ has rank equal to $\dim(\T^3\times\T^3)=6$, the two-point Markov chain $P^{(2)}$ will admit a small set. Notice that since on each time step the noise vector $\omega=(\A,\B,\C,\alpha,\beta,\gamma)\in\Omega_0$ is six dimensional, we can restrict ourselves to the case $n=1$. In particular we choose
\[
\xvec^1=(0,0,0),\qquad \xvec^2=\left(\frac{\pi}{2},\frac{\pi}{2},\frac{\pi}{2} \right), \qquad (\A,\B,\C)=(\pi,\pi,1), \qquad (\alpha,\beta,\gamma)=(0,0,0),
\]
so that there yields the matrix
\[
D_{\omega}\Phi^{(2)}_{(\xvec^1,\xvec^2)} = 
\begin{pmatrix}
0 & 0 & -1 & \pi & 0 & 0\\
1 & 0 & 0 & \pi^2 & \pi & 0\\
-1 & 1 & 0 & -\pi^2 & -\pi & -1\\
1 & -1 & 0 & -\pi & 0 & 1\\
0 & -1 & 0 & -\pi & 0 & 0\\
\pi & 0 & -1 & 0 & \pi & 0
\end{pmatrix}
\]
which has determinant equal to $2\pi^2$, and thus full rank.

Unlike for the projective and on-point chains, irreducibility $+$ aperiodicity $+$ small sets does not imply geometric ergodicity since the domain $(\T^3\times\T^3)\setminus\Delta$ is not compact. The Lyapunov-drift condition from Theorem \ref{thm:harris} is also needed, and in order to to prove that it is also satisfied we need first to deal with the top Lyapunov exponent and Lemma \ref{lemma:lyapunov-drif-condition}.

\subsection{Top Lyapunov exponent}\label{section:lyapunov}

The final ingredient missing for the main theorem is the positivity of the top Lyapunov exponent for the continuous random dynamical system $f_{\underline{\omega}}^n$ on $\T^3$.

In order to do so we will apply Proposition \ref{prop:lyapunov}, that give sufficient conditions to rule out the case $\lambda_1=0$ via Furstenberg's criterion, i.e.\ Proposition \ref{prop:furstenberg}. Observe that conditions \ref{item:L1} and \ref{item:L2} in Proposition \ref{prop:lyapunov} coincide with the conditions required for the small set from Proposition \ref{prop:small-set}, and we have already found $n\in\N$, $\underline{\omega}_\star^n\in\Omega_0^n$ and $\xvec^\star\in\T^3$ so that this holds for the one-point process in Section \ref{section:one-point-chain}. However, to prove that $f_{\underline{\omega}}^n$ has a positive tope Lyapunov exponent we must come up with a choice of $n\in\N$, $\underline{\omega}_\star^n\in\Omega_0^n$ and $\xvec^\star\in\T^3$ so that the three conditions \ref{item:L1}--\ref{item:L3} are satisfied simultaneously. 

Notice that $\dim SL_3(\R)=8$, since we need $\Phi_{\xvec^\star}(\underline{\omega}^n)$ to be a submersion we find that
\[
\dim \ker D_{\underline{\omega}^n_\star}\Phi_{\xvec}\leq 6n-3.
\]
Therefore we must choose $n\geq 2$ in order to be able to find parameters for which $D_{\underline{\omega}^n_\star}\Psi_{\xvec^\star}$ restricted to $\ker D_{\underline{\omega}^n_\star}\Phi_{\xvec^\star}$ is a surjection onto $T_{\Psi_{\xvec^\star}(\underline{\omega}^n)}SL_3(\R)$.

We choose $n=2$ and make the following selection of parameters
\[
\xvec^\star=(0,0,0),\quad (\A_1,\B_1,\C_1) = (\pi,\pi,\pi),\quad (\A_2,\B_2,\C_2)=(\pi,1,0),
\]
and
\[
(\alpha_1,\beta_1,\gamma_1) = \left(\frac{\pi}{2}, \frac{\pi}{2},\frac{\pi}{2}\right),\quad (\alpha_2,\beta_2,\gamma_2) = (0,0,0).
\]
Computations at this stage are long and tedious, so we present here a schematic version of the calculations for the convenience of the reader. We set $\underline{\omega}_\star^2=(\omega_1,\omega_2)$ with $\omega_1=(\A_1,\B_1,\C_1,\alpha_1,\beta_1,\gamma_1)$ and $\omega_2=(\A_2,\B_2,\C_2,\alpha_2,\beta_2,\gamma_2)$, so first of all we obtain the matrix
\setcounter{MaxMatrixCols}{12}
\[
D_{\underline{\omega}_\star^2}\Phi_{\xvec^\star} = 
\begin{pmatrix}
1-\pi^2 & 0 & -\pi & 0 & \pi & 1 & -\pi^2 & -\pi & -\pi^2 & 0 & \pi & 0\\
\pi^2-1 & -1 & \pi-1 & 0 & -\pi & 0 & \pi^2-\pi & \pi & \pi^2 & -1 & -\pi & 0\\
\pi & 0 & 0 & -1 & -1 & 0 & 0 & 0 & \pi & 0 & 0 & 0
\end{pmatrix}.
\]
Since the rank of $D_{\underline{\omega}_\star^2}\Phi_{\xvec^\star}$ is $3$, we see that this choice of parameters verifies property \ref{item:L2} from Proposition \ref{prop:lyapunov}. Moreover, property \ref{item:L1} is also verified by the definition of the noise in Section \ref{subsection:definition-flow}, and therefore there is only property \ref{item:L3} left to be checked.  

Observe that the kernel of $D_{\underline{\omega}_\star^2}\Phi_{\xvec^\star}$ is $9$ dimensional, and it is spanned by the columns of the following matrix
\[
\mathcal{K} = 
\begin{pmatrix}
  0 & 0 & 0 & -1 & 0 & 0 & 0 & \pi^{-1} & \pi^{-1} \\
0 & -1 & -1 & \pi^{-1} & 1 & 0 & 1-\pi^{-1} & -\pi^{-2} & 1-\pi^{-2} \\
0 & 1 & 0 & -\pi^{-1} & -1 & -\pi & \pi^{-1} & \pi^{-2} & -1+\pi^{-2} \\
0 & 0 & 0 & 0 & 0 & 0 & 0 & 0 & 1 \\
0 & 0 & 0 & 0 & 0 & 0 & 0 & 1 & 0 \\
0 & 0 & 0 & 0 & 0 & 0 & 1 & 0 & 0 \\
0 & 0 & 0 & 0 & 0 & 1 & 0 & 0 & 0 \\
0 & 0 & 0 & 0 & 1 & 0 & 0 & 0 & 0 \\
0 & 0 & 0 & 1 & 0 & 0 & 0 & 0 & 0 \\
0 & 0 & 1 & 0 & 0 & 0 & 0 & 0 & 0 \\
0 & 1 & 0 & 0 & 0 & 0 & 0 & 0 & 0 \\
1 & 0 & 0 & 0 & 0 & 0 & 0 & 0 & 0 \\  
\end{pmatrix}
\]
Next, we identify the space of $3\times 3$ real matrices with $\R^9$ via the standard parametrisation,
\[
(a_{ij})_{i,j=1}^3 \mapsto (a_{11},a_{12},a_{13},a_{21},a_{22},a_{23},a_{31},a_{32},a_{33})^\top.
\]
By means of this parametrisation, we introduce $\mathcal{A}=D_{\underline{\omega}_\star^2}\Psi_{\xvec^\star}$. Due to space limitations in this format, we decide to split the $9\times 12$ matrix $\mathcal{A}$ in two $9\times 6$ submatrices, namely $\mathcal{A}=(\mathcal{A}_1|\mathcal{A}_2)$. For our choice of parameters we find that
\[
\mathcal{A}_1 = 
\begin{pmatrix}
\pi^2 & -\pi & -\pi & 0 & 0 & 0 \\
0 & 0 & \pi^2 & 0 & 1 & 0 \\
-\pi & -1 & -\pi^3 & 0 & -\pi & 0 \\
\pi^3-\pi^2+\pi & \pi & \pi & \pi^2-1 & -\pi^2 & 0 \\
0 & 0 & -\pi^2 & -\pi & -1 & 0 \\
\pi^2+\pi-1 & 1 & \pi^3 & \pi^2+\pi & 0 & 0 \\
(\pi^2-1)^2 & 0 & \pi^3-\pi+1 & 0 & -\pi^3+\pi & \pi^2-1 \\
-\pi^3+\pi & 0 & -\pi^2-\pi & 0 & \pi^2 & -\pi+1 \\
\pi(\pi+1)^2(\pi-1) & 0 & \pi^2(\pi+2) & 0 & -\pi^2(\pi+1) & \pi^2 
\end{pmatrix},
\]
and
\[
\mathcal{A}_2 = 
\begin{pmatrix}
0 & 0 & \pi^2 & 0 & 0 & 0 \\
\pi^3 & 0 & 0 & 0 & -\pi^2 & 0 \\
-\pi^4+\pi^3-\pi & 0 & 0 & 0 & \pi^3 & 0 \\
0 & \pi^2 & \pi^3-\pi^2+\pi & 0 & 0 & 0 \\
-\pi^3 & 0 & 0 & 0 & \pi^2 & 0 \\
\pi^4-\pi^3+\pi & \pi & \pi^2 & 0 & -\pi^3 & 0 \\
\pi^2(\pi^2-1) & \pi^3-\pi & \pi^2(\pi^2-1) & -\pi^2+1 & -\pi^3+\pi & 0 \\
-\pi^3-\pi^2 & -\pi^2 & -\pi^3 & \pi & \pi^2+\pi & 0 \\
\pi^4+2\pi^3-\pi^2 & \pi^3+\pi^2 & \pi^4+\pi^3 & -\pi^2-\pi & -\pi^3-2\pi^2 & 0 
\end{pmatrix}.
\]
Hence, to sum up we need to show that $\mathcal{A}=(\mathcal{A}_1|\mathcal{A}_2)$ restricted to $\ker(D_{\underline{\omega}^n_\star}\Phi_{\xvec})$ as a linear operator has rank $\geq 8$. A straightforward computation yields that $\mathcal{A}\mathcal{K}$ is a $9\times 9$ matrix that we present here split in two as before $\mathcal{A}\mathcal{K}=((\mathcal{A}\mathcal{K})_1|(\mathcal{A}\mathcal{K})_2)$, where $(\mathcal{A}\mathcal{K})_1$ is a $9\times 5$ matrix, and $(\mathcal{A}\mathcal{K})_2$ is a $9\times 4$ matrix with the following form
\[
(\mathcal{A}\mathcal{K})_1 = 
\begin{pmatrix}
0 & 0 & \pi & 0 & 0 \\
0 & 0 & 0 & -\pi^2 & -\pi^2 \\
0 & 1 & 1 & \pi^2+\pi-\pi^{-1} & \pi^3-1  \\
0 & 0 & -\pi & 0 & \pi^2 \\
0 & 0 & 0 & \pi & \pi^2 \\
0 & -1 & -1 & \pi^{-1}(\pi-1)(\pi+1)^2 & -\pi^3+\pi+1 \\
0 & 1 & 1-\pi^2 & -\pi^{-1} & -1 \\
0 & 0 & \pi & 1 & \pi  \\
0 & 0 & -\pi^2-\pi & -\pi & -\pi^2
\end{pmatrix},
\]
and
\[
(\mathcal{A}\mathcal{K})_2 = 
\begin{pmatrix}
\pi^2 & -\pi & \pi & \pi\\
0 & \pi & 2 & -\pi^2+1\\
\pi^3-\pi & -\pi^2-1+\pi^{-1} & -2\pi-1+\pi^{-2} & \pi^3-\pi-2+\pi^{-2}\\
 -\pi^2 & \pi & -\pi+1 & 2\pi^2-\pi\\
0 & -\pi & -2 & \pi^2-\pi-1\\
 -\pi^3+\pi & \pi^2+1-\pi^{-1} & 2\pi+1-\pi^{-1}-\pi^{-2} & -\pi^3+\pi^2+3\pi+2-\pi^{-1}-\pi^{-2}\\
 -\pi & 2\pi^2-2+\pi^{-1} & \pi^{-2} & -1+\pi^{-2}\\
 0 & -2\pi & -\pi^{-1} & \pi-\pi^{-1}\\
-\pi^2 & 2\pi(\pi+1) & 1 & -\pi^2+1
\end{pmatrix}.
\]

The matrix $\mathcal{A}\mathcal{K}$ has rank $8$ since all columns are linearly independent, except the first trivial column in $(\mathcal{A}\mathcal{K})_1$. This completes the requirements needed to show, via Proposition \ref{prop:lyapunov}, that the top Lyapunov exponent is strictly positive.

\begin{proof}[Proof of Theorem \ref{thm:mixing}]
    We finally have all the bits needed to close the proof of Theorem \ref{thm:mixing}. By means of Proposition \ref{prop:discrete-mixing} we see that the goal is to show that the two-point chain $P^{(2)}$ is geometrically ergodic. To begin with notice that the basic assumptions \ref{item:H1}--\ref{item:H3} are trivially satisfied by the random ABC vector field \eqref{eq:random-abc-vectorfield} with absolutely continuous noise. In Sections \ref{section:one-point-chain} and \ref{section:projective-chain-2} we prove geometric ergodicity of the one-point and projective chains respectively. Additionally in Section \ref{sectio:two-point-chain} we demonstrate that the two-point chain satisfies the first three requirements from Harris theorem, namely irreducibility, aperiodicity and small sets. Finally in Section \ref{section:lyapunov} we show that the top Lyapunov exponent is strictly positive for almost all $\xvec\in\T^3$ and $\underline{\omega}\in\Omega$, which yields via Lemma \ref{lemma:lyapunov-drif-condition} that the two-poin chain $P^{(2)}$ satisfies the Lyapunov-drift condition. $P^{(2)}$ is geometrically ergodic, and hence the claim in Theorem \ref{thm:mixing} follows for all discrete times $t=n\in\N$. To include all the positive real times we write $t=n+\tau>0$ with $n\in\N$ and $\tau\in[0,1)$, then via Proposition \ref{prop:discrete-mixing} we obtain the estimate
    \[
    \left| \int_{\T^3} g(\xvec)h(\phi_{n+\tau}^\kappa(\xvec))\dd \xvec \right| \leq \widehat{C}_{\underline{\omega},\kappa}\|g\|_{H^s} \|h\circ \phi_\tau^\kappa\|_{H^s}\e^{-\lambda_s n}.
    \]
    We claim that we can control the $H^s$ norm of $h\circ \phi_s^\kappa$ by the $H^s$ norm of $h$. Using the fact that we are working in the torus we can write
    \[
    \|h\circ\phi_\tau^\kappa\|_{H^s}^2\lesssim \sum_{\kvec\in\Z^3_0} |\kvec|^{2s}|\widehat{h\circ\phi_\tau^\kappa}(\kvec)|^2 = \sum_{\kvec\in\Z^3_0} |\kvec|^{2s}|\widehat{h}(\kvec)|^2|\widehat{\phi_\tau^\kappa}(\kvec)|^2\leq \|\widehat{\phi_\tau^\kappa}\|_{L^\infty_{\kvec}}^2\|h\|_{H^s}^2,  
    \]
    where $\Z_0^3 = \Z^3\setminus\{(0,0,0)\}$. Now using the standard estimate $\|\widehat{\phi_\tau^\kappa}\|_{L^\infty_{\kvec}}\lesssim \|\phi_\tau^\kappa\|_{L^\infty_{\xvec}}\leq |\T^3|$, where $|\T^3|= (2\pi)^3$ denotes the size of the torus, we conclude that 
    \[
    \left| \int_{\T^3} g(\xvec)h(\phi_{n+\tau}^\kappa(\xvec))\dd \xvec \right| \leq \widehat{C}_{\underline{\omega},\kappa}(2\pi)^3\e^{\lambda_s}\|g\|_{H^s} \|h\|_{H^s}\e^{-\lambda_s (n+\tau)}, 
    \]
    so there follows the statement of Theorem \ref{thm:mixing} with $\widehat{D}_{\underline{\omega},\kappa} = \widehat{C}_{\underline{\omega},\kappa}(2\pi)^3\e^{\lambda_s}$.
\end{proof}

\section{Ideal kinematic dynamo: proof of Theorem \ref{thm:fast-dynamo}}\label{section:dynamo}

The last part of this paper is devoted to the ideal dynamo problem. The results presented in the previous sections are directly applicable to obtain the growth of a magnetic field that is advected and stretched by the random ABC vector field \eqref{eq:random-abc-vectorfield} via the ideal kinematic dynamo equation,
\[
\partial_t B + (u\cdot\nabla) B - (B\cdot \nabla) u  =  0, \quad \nabla\cdot B  =  0,
\]
with $t>0$, $\xvec\in\T^3$, and with initial datum $B(0,\cdot) = B_0\in L^1(\T^3)$.

When $\kappa=0$], solutions to \eqref{eq:dynamo} can be recovered from the flow map solving the ODE
\begin{equation}\label{eq:flowODE}
\ddt\phi_t(\xvec) = u(t,\phi_t(\xvec)),\quad \phi_0(\xvec) = \xvec, 
\end{equation}
via the formula (compare with \eqref{eq:FKpassiveV})
\begin{equation}\label{eq:B-characteristics}
B(t,\phi_t(\xvec)) = D_{\xvec}\phi_t B_0(\xvec).
\end{equation}
Indeed, writing \eqref{eq:B-characteristics} component-wise and taking a time derivative, for $i=1,2,3$ we find that 
\begin{align*}
(\partial_t B^i +u\cdot \nabla B^i)(t,\phi_t(\xvec))
&=\ddt B^i(t,\phi_t(\xvec))
=\ddt \left(\partial_j \phi_t^i(\xvec)B_0^j(\xvec)\right)
=\partial_j \left(u^i(t,\phi_t(\xvec))\right)B_0^j(\xvec)\\
&=\nabla u^i(t,\phi_t(\xvec)) \cdot \partial_j\phi_t(\xvec) B_0^j(\xvec)
=\partial_j \phi^\ell (\xvec) B_0^j(\xvec) \partial_\ell u^i (t,\phi_t(\xvec))\\
&=(B^\ell\cdot\partial_\ell u^i)(t,\phi_t(\xvec))
=(B\cdot\nabla u^i)(t,\phi_t(\xvec)),
\end{align*}
as needed.
In order to show that the ABC flow with random coefficients and random phases is an ideal kinematic dynamo, we will use two key features that we introduced for the mixing problem: the positivity of the top Lyapunov exponent, and the ergodicity of the projective process. This argument is inspired by the ideas used in \cite{BBPS22a} to prove growth of the gradient of the passive scalar in every $L^p$ space.

\begin{proposition}\label{prop:lambda-all-directions}
Assume that that the projective Markov transition kernel $\hat{P}$ has a unique stationary measure. Then for a.e.\ $\xvec\in\T^3$, $\Prob-$a.e.\ $\underline{\omega}\in\Omega$, and all $\vvec\in\S^2$ we find that
\[
\lim_{n\to \infty}\frac{1}{n}\log |D_{\xvec} f_{\underline{\omega}}^n\vvec| = \lambda_1.
\]
\end{proposition}

This proposition states that for almost every realisation, the linear cocycle $D_{\xvec}f_{\underline{\omega}}^n$ only sees the expanding direction associated to the top Lyapunov exponent. Proposition \ref{prop:lambda-all-directions} can be understood as a consequence of a random version of Oseledec's Multiplicative Ergodic Theorem introduced by Kifer, that can be found in \cite{Kifer86}*{Theorem III.1.2}. With these tools in hand, we find the following rather direct corollary.

\begin{corollary}\label{corollary:dynamo}
Under the assumptions of Proposition \ref{prop:lambda-all-directions} and hypotheses \ref{item:H1}-\ref{item:H3}, if $\lambda_1>0$, then for any $\varepsilon\in(0,\lambda_1)$, for a.e. $\xvec\in\T^3$, $\Prob-$a.s. $\underline{\omega}\in\Omega$ and all $\vvec\in \S^2$, there exists a positive constant $\hat{k}_1 = \hat{k}_1(\xvec,\underline{\omega},\vvec,\varepsilon)>0$ such that
\[
|D_{\xvec} f_{\underline{\omega}}^n\vvec| \geq \hat{k}_1 \e^{(\lambda_1-\varepsilon)n},
\]
for all $n\in \N$.
\end{corollary}

\begin{proof}
By Proposition \ref{prop:lambda-all-directions}, given $\xvec\in\T^3$, $\underline{\omega}\in\Omega$ and $\vvec\in\S^2$, one can choose $\varepsilon\in (0,\lambda_1)$ such that there exists $N=N(\xvec,\underline{\omega},\vvec,\varepsilon)\in\N$ for which
\[
\left| \frac{1}{n}\log |D_{\xvec} f_{\underline{\omega}}^n\vvec| - \lambda_1\right| <\varepsilon
\]
holds true for every $n\geq N$. In particular, we obtain that $\log |D_{\xvec} f_{\underline{\omega}}^n| > n(\lambda_1-\varepsilon)$, or equivalently
\[
|D_{\xvec} f_{\underline{\omega}}^n\vvec| > \e^{(\lambda_1-\varepsilon)n},
\]
for all $n\geq N$. To deal with $n<N$ we use a classical result, see for instance \cite{YuGu97}, that states that for any matrix $M\in\R^{d\times d}$ and vector $\wvec\in\R^d$ there holds the bound
\[
|M\wvec| \geq (d-1)^{(d-1)/4}\frac{|\det M|}{\|M\|_F^{d-1}}|\wvec|,
\]
where $\|\cdot\|_F$ denotes the Frobenius norm. In our specific setting, notice that we have $\det (D_{\xvec} f_{\underline{\omega}}^n)=1$ for all choices of parameters, and $\|D_{\xvec} f_{\omega}\|_F^2 \leq K$ for some $K\geq 1$ that only depends on $U$, the upper bound of the parameters $|\A|,|\B|,|\C|\leq U$. Hence, for any fixed $\xvec\in\T^3$, $\underline{\omega}\in\Omega$ and $\vvec\in\S^2$ we find that 
\[
|D_{\xvec} f_{\underline{\omega}}^n\vvec| \geq \frac{\sqrt{2}}{K^n}.
\]
To conclude just notice that for all $n\geq N$ there holds $|D_{\xvec} f_{\underline{\omega}}^n\vvec| > \e^{(\lambda_1-\varepsilon)n}$, while for all $n<N$ we obtain $|D_{\xvec} f_{\underline{\omega}}^n\vvec| \geq \sqrt{2}K^{-N}$. Therefore there exists a constant $\hat{k}_1=\hat{k}_1(N)>0$ for which the claim of the corollary holds true.
\end{proof}

\begin{proof}[Proof of Theorem \ref{thm:fast-dynamo}]
To prove Theorem \ref{thm:fast-dynamo} we need to compute the $L^1$ norm of the magnetic field $B(t,\xvec)$ and apply Corollary \ref{corollary:dynamo}. The growth of all the rest $L^p$ norms follows by Hölder since the domain is bounded. First of all, notice that since the vector field \eqref{eq:random-abc-vectorfield} is divergence-free, and hence the flow is volume preserving, there holds that
\[
\|B(t)\|_{L^1} = \int_{\T^3} |B(t,\xvec)| \dd\xvec = \int_{\T^3} |B(t,\phi_t(\xvec))| \dd (\phi_t)_\#\xvec =  \int_{\T^3} |B(t,\phi_t(\xvec))| \dd \xvec.
\]
Then, using the convenient representation of $B$ in terms of the flow \eqref{eq:B-characteristics},
\[
\|B(t)\|_{L^1} = \int_{\T^3} |D_{\xvec}\phi_t B_0(\xvec)| \dd\xvec.
\]
Notice that if we choose $t=n\in \N$, then the flow can be written as $\phi_t(\xvec)=f_{\underline{\omega}}^n(\xvec)$ and we are in disposition to apply Corollary \ref{corollary:dynamo} for discrete times $t=n\in\N$. Fix $\varepsilon>0$ and $B_0\in L^1(\T^3)$ with $\|B_0\|_{L^1}>0$. By Corollary \ref{corollary:dynamo}, there exists an almost surely positive constant $\hat{k}_1$ depending on $(\xvec,\underline{\omega})\in\T^3\times\Omega$, $\varepsilon>0$ and $B_0\in L^1$ such that for all natural times $n\in\N$ we obtain
\[
\begin{split}
\|B(n)\|_{L^1} & \geq \int_{\{\xvec\in\T^3 :|B_0(\xvec)|\neq 0\}} \left|D_{\xvec}\phi_t \frac{B_0(\xvec)}{|B_0(\xvec)|}\right||B_0(\xvec)|\dd\xvec \\
& \geq \e^{(\lambda_1-\varepsilon)n} \int_{\{\xvec\in\T^3 :|B_0(\xvec)|\neq 0\}} \hat{k}_1\left(\xvec,\underline{\omega},\frac{B_0(\xvec)}{|B_0(\xvec)|},\varepsilon\right)|B_0(\xvec)|\dd\xvec.
\end{split}
\]
Since $\hat{k}_1>0$ for almost every $\xvec$ and $\underline{\omega}$, and since we assume $B_0\not\equiv 0$, then
\[
\hat{k}_2(\underline{\omega},B_0,\varepsilon) = \int_{\{\xvec\in\T^3 :|B_0(\xvec)|\neq 0\}} \hat{k}_1\left(\xvec,\underline{\omega},\frac{B_0(\xvec)}{|B_0(\xvec)|},\varepsilon\right)|B_0(\xvec)|\dd\xvec
\]
is $\Prob-$almost surely positive.

Finally, a standard stability estimate covers all times in between $n$ and $n+1$, up to a new constant $\hat{k}$ (almost surely positive) depending exclusively on $\hat{k}_2$, $\lambda_1$, $\varepsilon$, and the Lipschitz constant of $u$. 
\end{proof}

\begin{remark}
Observe that the proof of ideal kinematic dynamo only requires a vector field with a flow satisfying Hypothesis \ref{item:H1}-\ref{item:H3}, and with positive Lyapunov exponent and ergodic projective Markov chain. The ABC vector field \eqref{eq:abc} provides an example of such flow, but this is a general proof covering all vector fields with these properties.
\end{remark}

\addtocontents{toc}{\protect\setcounter{tocdepth}{0}}

\section*{Acknowledgement}

\addtocontents{toc}{\protect\setcounter{tocdepth}{1}}

MCZ gratefully acknowledges support by the Royal Society URF\textbackslash R1\textbackslash 191492 and the ERC/EPSRC Horizon Europe Guarantee EP/X020886/1. VNF gratefully acknowledges support by the ERC/EPSRC Horizon Europe Guarantee EP/X020886/1.
We would like to thank David Villringer for illuminating discussions and insightful comments.

\bibliographystyle{abbrv}
\bibliography{euler.bib}
\end{document}

%% file: header.tex
\usepackage{amsmath,amsfonts,amssymb,amsthm}
\usepackage[abbrev,lite,nobysame]{amsrefs}
\usepackage{mathrsfs}
\usepackage{graphics,graphicx}
\usepackage[usenames,dvipsnames]{xcolor}
\usepackage{times}
\usepackage{bbm}
\usepackage[margin=1in]{geometry}

\usepackage{enumitem}
\usepackage[colorlinks=true, pdfstartview=FitV, linkcolor=BrickRed,citecolor=black, urlcolor=black]{hyperref}

\usepackage[centercolon]{mathtools}
\newcommand\dd{{\rm d}}
\newcommand\e{{\rm e}}
\newcommand\ddt{{\frac{\dd}{\dd t}}}

\newcommand{\N}{\mathbb{N}}
\newcommand{\Z}{\mathbb{Z}}

\newcommand{\R}{\mathbb{R}}

\newcommand{\T}{\mathbb{T}}
\newcommand{\cF}{{\mathcal F}}

\renewcommand{\S}{\mathbb{S}}
\newcommand{\Bo}{\mathcal{B}}
\newcommand{\Prob}{\mathbb{P}}
\newcommand{\EX}{\mathbb{E}}

\newcommand{\A}{\mathsf{A}}
\newcommand{\B}{\mathsf{B}}
\newcommand{\C}{\mathsf{C}}

\newcommand{\xvec}{\boldsymbol{x}}
\newcommand{\zvec}{\boldsymbol{z}}
\newcommand{\vvec}{\boldsymbol{v}}
\newcommand{\wvec}{\boldsymbol{w}}
\newcommand{\kvec}{\boldsymbol{k}}

\DeclareMathOperator{\supp}{supp}

\newcommand{\cP}{\mathcal{P}}

\newcommand{\dist}{\text{\rm dist}}

\newtheorem{theorem}{Theorem}
\newtheorem{proposition}{Proposition}[section]
\newtheorem{lemma}[proposition]{Lemma}
\newtheorem{corollary}[proposition]{Corollary}

\theoremstyle{definition}
\newtheorem{definition}[proposition]{Definition}

\newtheorem{remark}[proposition]{Remark}

\numberwithin{equation}{section}

\newcommand{\cB}{\mathcal{B}}
\newcommand{\F}{\mathcal{F}}

\newcommand{\X}{\mathcal{X}}

\numberwithin{equation}{section}

\makeatletter

\renewcommand\subsubsection{\@startsection{subsubsection}{3}%
\normalparindent{.5\linespacing\@plus.7\linespacing}{-.5em}
{\normalfont\bfseries}}

\def\@tocline#1#2#3#4#5#6#7{\relax
  \ifnum #1>\c@tocdepth % then omit
  \else
    \par \addpenalty\@secpenalty\addvspace{#2}%
    \begingroup \hyphenpenalty\@M
    \@ifempty{#4}{%
      \@tempdima\csname r@tocindent\number#1\endcsname\relax
    }{%
      \@tempdima#4\relax
    }%
    \parindent\z@ \leftskip#3\relax \advance\leftskip\@tempdima\relax
    \rightskip\@pnumwidth plus4em \parfillskip-\@pnumwidth
    #5\leavevmode\hskip-\@tempdima
      \ifcase #1
       \or\or \hskip 1em \or \hskip 2em \else \hskip 3em \fi%
      #6\nobreak\relax
    \dotfill\hbox to\@pnumwidth{\@tocpagenum{#7}}\par
    \nobreak
    \endgroup
  \fi}
\makeatother